\pdfoutput=1
\documentclass[10pt]{article}

\usepackage[a4paper, left=3cm, right=3cm, top=3cm, bottom=2.6cm]{geometry}
\usepackage{lmodern}
\usepackage{microtype}
\usepackage[onehalfspacing]{setspace}
\usepackage[T1]{fontenc}
\usepackage[utf8]{inputenc}
\usepackage[english]{babel}
\usepackage{subcaption}
\usepackage{amsmath,amssymb,amsthm}
\usepackage{mathtools}

\usepackage{graphicx}
\usepackage{booktabs}
\usepackage{array}
\usepackage{float}
\usepackage{tikz}
\usetikzlibrary{arrows.meta,positioning,calc,patterns}

\usepackage{enumitem}
\setlist[enumerate]{nosep}
\setlist[itemize]{nosep}

\usepackage{xcolor}

\newcommand*{\R}{\mathbb{R}}

\DeclareMathOperator*{\argmin}{arg\,min}
\DeclareMathOperator*{\range}{ran}
\newcommand{\id}{\operatorname{Id}}

\newcommand\set[1]{\left\{#1\right\}}

\DeclareMathOperator{\Equi}{Equi}
\DeclareMathOperator{\Fac}{Fac}

\newcommand{\E}{\mathbb{E}}

\newcommand{\Esq}[1]{\mathbb{E}\bigl[\lVert #1 \rVert_2^2\bigr]}

\DeclarePairedDelimiter{\abs}{\lvert}{\rvert}
\DeclarePairedDelimiter{\norm}{\lVert}{\rVert}

\newcommand{\Bop}{B}
\newcommand{\Eop}{E}

\newcommand{\T}{\mathcal{T}}
\newcommand{\TS}{\mathcal{S}}
\newcommand{\noise}{\epsilon}

\newcommand{\loss}{\mathcal{L}}

\newif\ifanonymous
\anonymousfalse

\theoremstyle{plain}
\newtheorem{theorem}{Theorem}[section]
\newtheorem{lemma}[theorem]{Lemma}
\newtheorem{proposition}[theorem]{Proposition}
\newtheorem{corollary}[theorem]{Corollary}
\newtheorem{condition}[theorem]{Condition}

\theoremstyle{definition}
\newtheorem{definition}[theorem]{Definition}

\theoremstyle{remark}
\newtheorem{remark}[theorem]{Remark}

\definecolor{elippsblue}{HTML}{1F4E79}
\definecolor{elippsaccent}{HTML}{C0504D}
\definecolor{elippsgray}{HTML}{F2F2F2}
\definecolor{elippsmid}{HTML}{9DC3E6}

\numberwithin{equation}{section}
\numberwithin{figure}{section}
\numberwithin{table}{section}

\usepackage{authblk}

\usepackage{fancyhdr}
\fancypagestyle{plain}{\fancyhf{}}

\usepackage[colorlinks=true,
            linkcolor=elippsblue,
            citecolor=elippsblue,
            urlcolor=elippsblue,
            pdftitle={ELIPPS: Exact Learning for Inverse Problems from Partial Self-supervision},
            pdfauthor={Benjamin Walder, Markus Haltmeier, Lukas Neumann, Nadja Gruber, Gyeongha Hwang}]{hyperref}

\title{\bfseries ELIPPS: Exact Learning for Inverse Problems\\[2pt] from Partial Self-supervision}

\author[1]{Benjamin Walder}
\author[1]{Markus Haltmeier}
\author[2]{Lukas Neumann}
\author[3]{Nadja Gruber}
\author[4]{Gyeongha Hwang\thanks{Corresponding author: \href{mailto:ghhwang@yu.ac.kr}{ghhwang@yu.ac.kr}}}

\affil[1]{Department of Mathematics, University of Innsbruck, Innsbruck, Austria}
\affil[2]{Institute of Basic Sciences in Engineering Science, University of Innsbruck, Innsbruck, Austria}
\affil[3]{Department of Computer Science, University of Innsbruck, Innsbruck, Austria}
\affil[4]{Department of Mathematics, Yeungnam University, Gyeongsan, Republic of Korea}

\date{\normalsize September 2026}

\begin{document}
\maketitle
\thispagestyle{plain}

\begin{abstract}
\noindent
In undersampled inverse problems (such as sparse-view computed tomography), only a small number of measurements are collected, which reduces radiation exposure and acquisition time and cost, and can also address the inaccessibility of certain acquisition arrangements. Most learning-based methods for such problems require supervision in the form of fully sampled measurements and ground-truth images, which are costly or even infeasible to acquire. To overcome this issue, we propose \emph{Exact Learning for Inverse Problems from Partial Self-supervision} (ELIPPS), an incomplete self-supervised training paradigm for undersampled inverse problems that requires neither ground-truth images nor fully sampled measurements. ELIPPS learns solely from incomplete forward measurements on a fixed incomplete supervision set. Our theory shows that when the data distribution is invariant under certain transformations, minimizing a masked empirical risk is equivalent to minimizing the full self-supervised risk, i.e. the risk against the complete, noise-free measurement. The equivalence is exact for arbitrary equivariant hypothesis classes and holds for signal-dependent noise such as the pre-log Poisson statistics of low-dose tomography, and for the log-transformed count model up to a quantifiable bias. When the underlying coverage condition is only approximately satisfied on a discrete grid, we give a stability estimate in terms of the associated frame constants and the irreducible error of the inverse problem. We realize ELIPPS for computed tomography by exploiting rotation and reflection invariance. In our experiments, ELIPPS substantially outperforms naive masked supervision and reaches the same order of accuracy as a reference model trained with full clean measurements.
\end{abstract}



\vspace{1em}

\section{Introduction}\label{sec:intro}

We study undersampled discrete inverse problems in which an unknown signal $x \in \mathbb{R}^N$ is observed through incomplete and noisy indirect measurements. Let $A:\mathbb{R}^N \to \mathbb{R}^M$ denote the full forward operator. For $I \subseteq \{1,\dots,M\}$, let $P_I:\mathbb{R}^M \to \mathbb{R}^{I}$ be the subsampling operator that selects the coordinates indexed by $I$, defined by $P_I z = (z_i)_{i \in I}$. Denote the subsampled forward operator by $A_I := P_I A$. The undersampled inverse problem reads
\begin{equation}\label{eq:ip}
  y_I = A_I x + \noise_I,
\end{equation}
where $\noise_I \in \mathbb{R}^{I}$ is additive measurement noise. The objective is to recover the underlying image $x$ from the incompletely sampled measurements $y_I \in \R^I$. Problem \eqref{eq:ip} is challenging due to both the ill-posedness of inverting $A$ and the incompleteness of the sampling process. 

A prominent example for \eqref{eq:ip} is sparse-view computed tomography (CT), where the full forward operator $A$ corresponds to the Radon transform mapping an image to its projection measurements (sinogram). In practical CT systems, measurements are often incomplete or limited due to sparse-view acquisition, detector subsampling, or low-dose scanning, which are commonly employed to reduce radiation exposure and acquisition time. These constraints make the reconstruction problem severely ill-posed and can introduce strong artifacts when classical reconstruction methods are applied directly.

\paragraph{Self-supervised reconstruction:}
Learning-based reconstruction methods have achieved remarkable performance for inverse problems, including \eqref{eq:ip}. Most approaches are supervised and assume access to paired training data $(y_I,x)$. In practice, collecting $x$ is challenging; typically even during training one only has access to additional measurements $y_S = A_S x + \noise_S$, where $S \subseteq \{1, \dots, M\}$ is the supervision set and $A_S = P_S A$. Inspired by self-supervised learning for denoising \cite{Batson2019,krull2019noise2void,moran2020noisier2noise}, several self-supervised strategies have been proposed that use only measurements of $y_S$. Self-supervised reconstruction methods can be roughly classified into: (i) methods with full self-supervision ($A_S = A$; at least in expectation), e.g., \cite{hendriksen2020noise2inverse,bubba2025tomoselfdeq,millard2023theoretical}; and (ii) methods with sparse self-supervision ($A_S=A_I$), e.g., \cite{gruber2025noisier2inverse,gruber2024sparse2inverse,unal2024proj2proj,schut2026equivariance2inverse}. The setting studied in this paper is a third class, \emph{partial self-supervision}, in which $A_S = P_S A$ for a supervision set $S$ that is fixed across acquisitions. Full self-supervision requires additional scans, higher radiation dose, or sensing configurations unavailable at deployment. Training with sparse self-supervision suffers from incomplete coverage of the data distribution, which can induce non-uniqueness and reliance on implicit priors. A recurring mechanism in this literature is measurement splitting: the acquired data are divided
into an input part and a target part, and the network is trained to predict the latter from the
former. This is the idea behind self-supervised reconstruction in MRI \cite{Yaman2020} and its
analysis via variable-density resampling \cite{millard2023theoretical}, and it is the
measurement-domain counterpart of training against a second noisy realization as in Noise2Noise
\cite{Lehtinen2018}.  What distinguishes our setting from all of the
above is that the split is neither random nor varying: the index sets $I$ and $S$ are fixed once and
for all by the acquisition.

A different line of research that leverages invariance is equivariant imaging (EI) \cite{chen2021equivariant,chen2022robust}, which learns a reconstruction from the undersampled
measurements $y_I$ by enforcing equivariance of the measurement--reconstruction pipeline through an
additional training penalty.  Equivariant splitting \cite{sechaud2025equivariant} uses invariance of the signal distribution to
reinterpret a single measurement as a measurement of a transformed signal through a transformed
operator; see \cite{tachella2026selfsup} for a comprehensive account of this field. Our setting differs from this line of work in two main respects. First, we assume partial self-supervision \eqref{eq:partial} on a \emph{fixed, deterministic} supervision set $S$, as opposed to unsupervised losses or random measurement splits, whose analysis requires sufficient diversity of the sampling patterns. Fixed  $S$ reflects acquisition designs in which the retained measurement subset is determined by the hardware and identical across samples.  Second, we do not promote equivariance through a penalty term but enforce it exactly by architectural symmetrization. Under a frame-type coverage condition we prove that the partially self-supervised risk and the full self-supervised risk are equivalent (Theorem~\ref{thm:main}), and we analyse when restricting the hypothesis class in this way costs nothing.

\paragraph{Partial self-supervision:}
We propose an incomplete self-supervised setting for \eqref{eq:ip}. For a fixed set $S \subseteq \set{1, \dots, M}$ we assume training data
\begin{equation}
\label{eq:partial}
   (y_I, y_S) =
   \bigl(A_I x + \noise_I,\, A_S x + \noise_S \bigr),
\end{equation}
i.e., pairs of undersampled measurements $y_I$ on the fixed set $I$ as in \eqref{eq:ip} together with an incomplete measurement $A_S x + \noise_S$ on the fixed set $S$. Such partial self-supervision arises when the retained subset of measurements is fixed by the acquisition
design. In CT practice this is the case, for instance, when the set of projection angles is
restricted by hardware, acquisition-time or dose constraints, or when a fixed, deliberately chosen
part of the detector is read out at full resolution while the remainder is not. In all of these
cases fully sampled reference measurements are unavailable even at training time, so that neither
supervised learning nor full self-supervision is applicable.

\definecolor{elippsorange}{HTML}{C55A11}
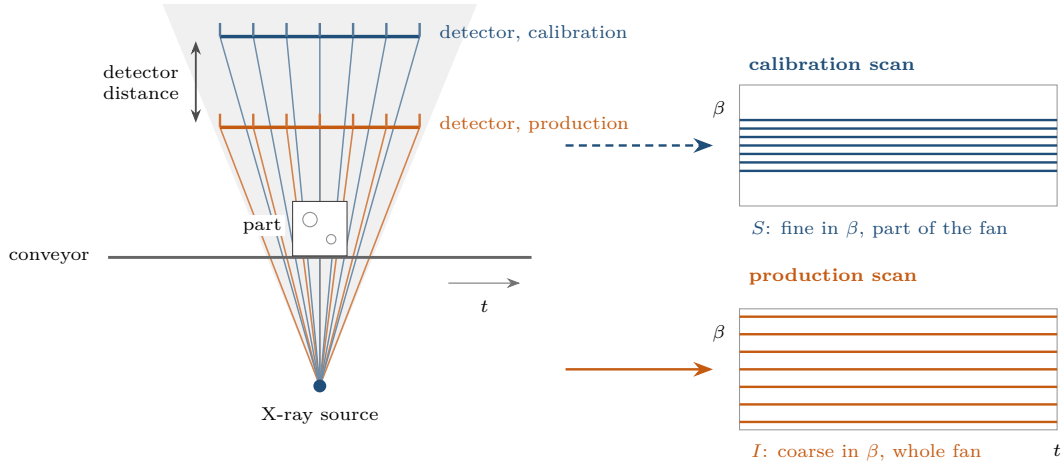
\begin{figure}[htb!]
\centering
\begin{tikzpicture}[
  font=\small,
  lab/.style    = {font=\scriptsize},
  hdr/.style    = {font=\scriptsize\bfseries},
  ray/.style    = {draw=black!28, line width=0.3pt},
  bigarr/.style = {-{Stealth[length=2.6mm]}, line width=0.9pt}
]

\begin{scope}[shift={(0,0)}]
\def\Sx{2.50}\def\Sy{-1.70}

\fill[black!6] (\Sx,\Sy) -- (0.42,3.35) -- (4.58,3.35) -- cycle;

\foreach \x in {1.18,1.62,...,3.82}{\draw[draw=elippsorange!75, line width=0.6pt] (\Sx,\Sy) -- (\x,1.72);}

\foreach \x in {1.18,1.62,...,3.82}{\draw[draw=elippsblue!65, line width=0.55pt] (\Sx,\Sy) -- (\x,2.92);}

\draw[line width=1.3pt, draw=elippsorange] (1.18,1.72) -- (3.82,1.72);
\foreach \x in {1.18,1.62,...,3.82}{\draw[draw=elippsorange!85, line width=0.9pt] (\x,1.72) -- (\x,1.90);}
\node[lab, anchor=west, text=elippsorange] at (3.96,1.76) {detector, production};

\draw[line width=1.3pt, draw=elippsblue] (1.18,2.92) -- (3.82,2.92);
\foreach \x in {1.18,1.62,...,3.82}{\draw[draw=elippsblue!75, line width=0.9pt] (\x,2.92) -- (\x,3.10);}
\node[lab, anchor=west, text=elippsblue] at (3.96,2.96) {detector, calibration};

\draw[{Stealth[length=1.8mm]}-{Stealth[length=1.8mm]}, black!70, line width=0.7pt]
      (0.86,1.78) -- (0.86,2.86);
\node[lab, anchor=east, align=right] at (0.76,2.32) {detector\\distance};

\draw[line width=1.1pt, draw=black!60] (-0.30,0) -- (5.30,0);
\node[lab, anchor=east] at (-0.42,0) {conveyor};
\draw[-{Stealth[length=1.8mm]}, black!55] (4.20,-0.34) -- (5.15,-0.34);
\node[lab, anchor=north] at (4.68,-0.44) {$t$};

\begin{scope}[shift={(2.50,0)}]
  \def\a{0.72}
  \draw[fill=white, draw=black!75] (-\a/2,0.02) rectangle (\a/2,\a+0.02);
  \draw[draw=black!45] (-0.13,0.50) circle (0.095);
  \draw[draw=black!45] (0.15,0.24) circle (0.062);
\end{scope}
\node[lab, anchor=east, fill=white, inner sep=1pt] at (2.04,0.42) {part};

\fill[elippsblue] (\Sx,\Sy) circle (0.085);
\node[lab, anchor=north] at (\Sx,\Sy-0.16) {X-ray source};

\end{scope}

\def\Wd{4.2}\def\Ht{0.80}
\begin{scope}[shift={(8.05,0)}]

\begin{scope}[shift={(0,1.48)}]
  \draw[draw=black!40] (0,-\Ht) rectangle (\Wd,\Ht);
  \foreach \k in {-3,...,3}{\draw[draw=elippsblue, line width=0.9pt]
      (0,\k*0.14*\Ht) -- (\Wd,\k*0.14*\Ht);}
  \node[lab, anchor=north west, text=elippsblue] at (0.03,-\Ht-0.07)
       {$S$: fine in $\beta$, part of the fan};
  \node[lab, anchor=south east] at (-0.06,0.30*\Ht) {$\beta$};
\end{scope}
\node[hdr, anchor=south west, text=elippsblue] at (0,2.36) {calibration scan};

\begin{scope}[shift={(0,-1.48)}]
  \draw[draw=black!40] (0,-\Ht) rectangle (\Wd,\Ht);
  \foreach \k in {-3,...,3}{\draw[draw=elippsorange, line width=0.9pt]
      (0,\k*0.29*\Ht) -- (\Wd,\k*0.29*\Ht);}
  \node[lab, anchor=north west, text=elippsorange] at (0.03,-\Ht-0.07)
       {$I$: coarse in $\beta$, whole fan};
  \node[lab, anchor=south east] at (-0.06,0.30*\Ht) {$\beta$};
  \node[lab, anchor=north] at (\Wd,-\Ht-0.07) {$t$};
\end{scope}
\node[hdr, anchor=south west, text=elippsorange] at (0,-0.48) {production scan};

\end{scope}

\draw[bigarr, densely dashed, draw=elippsblue] (5.75,1.48) -- (7.70,1.48);
\draw[bigarr, draw=elippsorange]               (5.75,-1.48) -- (7.70,-1.48);

\end{tikzpicture}
\caption{In-line industrial CT. Source and detector are fixed and the part is carried past them on
a conveyor, so that the belt position $t$ takes the role of the view direction and the position
$\beta$ on the detector line that of the detector coordinate. Production (solid): detector close,
coarse in $\beta$ over the whole fan, giving $y_I$. Calibration (dashed): detector moved back,
fine in $\beta$ over part of the fan, giving $y_S$.}
\label{fig:inline}
\end{figure}

As a concrete application, consider in-line industrial CT as illustrated in Figure~\ref{fig:inline}. Here $I$ and $S$ are separated by the acquisition itself rather than by a choice of dose. Source and
detector are fixed and the part is carried past them on a conveyor, so that the belt position takes
the role of the view direction and the position on the detector line that of the detector coordinate. Moving the detector away from the source leaves the number of detector elements
unchanged but narrows the fan they cover, so that the same line samples a smaller part of the fan
more finely. In production the detector sits close, giving a coarse measurement $y_I$ over the whole
fan; for calibration or first-article inspection it is moved back once, giving a fine measurement
$y_S$ over part of it. Both sets are fixed by the machine settings and neither contains the other.
Parts that pass through both settings provide the training pairs \eqref{eq:partial}; at deployment
only $y_I$ exists. Parts arrive in arbitrary in-plane orientation, so that the invariance of the
signal distribution is a property of the application rather than one enforced by augmentation.

Learning under partial self-supervision is underdetermined: a regression loss that penalizes
discrepancies only on the observed subset defined by $P_S$ leaves the other coordinates
unconstrained. As a result, models trained with naive self-supervision produce reconstructions that
are locally consistent but globally inconsistent or artifact-prone. While \cite{Walder2025} showed
that noise-free inpainting ($A=\id$ and $\noise_I=0$) can be solved from incomplete supervision
\eqref{eq:partial}, by exploiting equivariance of the restoration map that fills in the missing
data, the general case of incomplete, indirect observations with noise remains open. The step from
that setting to the present one is not merely one of generality. With $A = \id$ the measurement and
image domains coincide and the initial reconstruction is the identity, so that neither the interplay
between transformations of the two domains nor the question of whether the admitted class of
image-domain restoration maps is rich enough arises. Moreover, with $A$ ill-posed, noise is a major source of difficulty:  the amplification of small
perturbations makes the inverse problem challenging, and a noise-free model would miss this issue. The
target of the masked regression is then no longer the clean signal: without noise the conditional
expectation coincides with $X$ wherever the observation determines it, with noise it does not.

\paragraph{Main contributions:}
We propose ELIPPS for \eqref{eq:ip} with incomplete self-supervised data \eqref{eq:partial}. Our approach leverages invariance properties of the signal class together with intertwining relations of the forward map $A$ to enforce equivariance of the learned reconstruction with respect to known transformations.  The key result shows that, over an equivariant function class, minimizing the partially
self-supervised risk $\Esq{P_S \, g(Y_I) - Y_S}$ is equivalent to minimizing the full self-supervised risk $\Esq{g(Y_I) - AX}$. This enables globally consistent learning from partial self-supervision. We instantiate the approach for CT using detector upscaling (subsampling along the detector axis) and sparse-view reconstruction (subsampling along the angular axis).

Specific results:
(i) A framework built on invariances of the signal class and intertwining relations of the forward operator, with a proof that masked regression over an equivariant class is equivalent to the full self-supervised risk (Theorem~\ref{thm:main}).  ((ii) A characterization of the admissible supervision sets: a coverage condition on the transformed copies of $S$ is necessary as well as sufficient for the two risks to be proportional, and it bounds how small $S$ can be by the number of transformations and by the stabilizers of the action (Propositions~\ref{prop:orbit} and~\ref{prop:characterization}). (iii) A network construction that enforces equivariance by design and applies to any modality whose
forward operator intertwines with the transformations (Theorem~\ref{thm:Xnet}), with a proof that
this restriction does not change the minimizer (Theorem~\ref{thm:arch},
Corollary~\ref{cor:equivariant}), a counterexample showing which property of the initial
reconstruction cannot be dropped (Remark~\ref{rem:counterexample}), and an exact expression for what is lost when that property fails (Proposition~\ref{prop:price}).
 (iv) An exactness result for arbitrary equivariant hypothesis classes, which separates the correctness of the training criterion from the expressivity of the architecture (Corollary~\ref{cor:subclass}), and a stability estimate for an only approximately satisfied coverage condition (Theorem~\ref{thm:stability}). (v) The combination of these results into the statement used in practice, namely that the minimizer over the trained architecture class is the conditional expectation (Corollary~\ref{cor:elipps}) and that the learned reconstruction agrees with the minimum mean squared error estimator up to $\ker(A)$ (Corollary~\ref{cor:image}), together with a practical training objective and implementation (Section~\ref{sec:summary}, equation~\eqref{eq:loss_elipps2}), realized for CT under low-dose Poisson noise and severe undersampling, reaching the accuracy of a reference model trained with full clean measurements (Section~\ref{sec:experiments}, Table~\ref{tab:combined_results}). (vi) An ablation on the role of the symmetry of the training distribution (Section~\ref{sec:ablation}, Table~\ref{tab:combined_results-NoAug}).

Figure~\ref{fig:method} summarizes the training setup and the reconstruction pipeline.

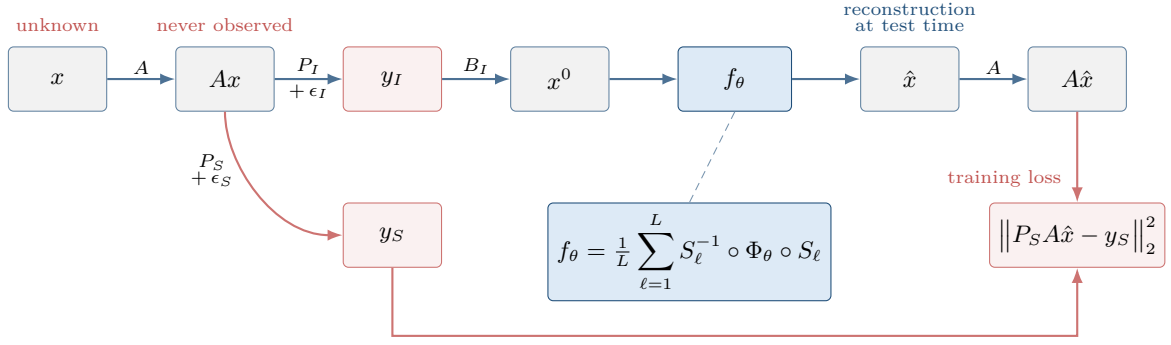
\begin{figure}[t]
\centering
\begin{tikzpicture}[
  font=\small,
  box/.style     = {draw=elippsblue!70, fill=elippsgray, rounded corners=2pt,
                    minimum height=8.5mm, minimum width=13mm, inner sep=3pt, align=center},
  net/.style     = {draw=elippsblue, fill=elippsmid!35, rounded corners=2pt,
                    minimum height=8.5mm, minimum width=15mm, align=center},
  data/.style    = {draw=elippsaccent!80, fill=elippsaccent!8, rounded corners=2pt,
                    minimum height=8.5mm, minimum width=13mm, align=center},
  ar/.style      = {-{Latex[length=2mm]}, draw=elippsblue!80, thick},
  lab/.style     = {font=\scriptsize, inner sep=1.5pt},
  node distance = 7mm and 9mm
]

\node[box] (x) {$x$};
\node[box, right=of x] (Ax) {$Ax$};
\node[data, right=of Ax] (yI) {$y_I$};
\node[box, right=of yI] (x0) {$x^{0}$};
\node[net,  right=of x0] (f) {$f_\theta$};
\node[box,  right=of f] (xhat) {$\hat x$};
\node[box,  right=of xhat] (Axhat) {$A\hat x$};

\draw[ar] (x)  -- node[lab, above] {$A$} (Ax);
\draw[ar] (Ax) -- node[lab, above] {$P_I$} node[lab, below] {$+\,\noise_I$} (yI);
\draw[ar] (yI) -- node[lab, above] {$B_I$} (x0);
\draw[ar] (x0) -- (f);
\draw[ar] (f)  -- (xhat);
\draw[ar] (xhat) -- node[lab, above] {$A$} (Axhat);

\node[data, below=12mm of yI] (yS) {$y_S$};
\node[box, draw=elippsaccent!80, fill=elippsaccent!8,
      below=12mm of Axhat, minimum width=22mm] (loss)
      {$\bigl\| P_S A\hat x - y_S \bigr\|_2^2$};

\draw[ar, draw=elippsaccent!80] (Ax.south) .. controls +(0,-7mm) and +(-9mm,0) ..
      node[lab, pos=0.35, left=0.5mm, align=center] {$P_S$\\[-2pt] $+\,\noise_S$} (yS.west);
\draw[ar, draw=elippsaccent!80] (yS.south) -- ++(0,-9mm) -| (loss.south);
\draw[ar, draw=elippsaccent!80] (Axhat) -- (loss);

\node[net, below=12mm of f, minimum width=34mm, xshift=-6mm] (sym)
      {$\displaystyle f_\theta=\tfrac{1}{L}\sum_{\ell=1}^{L} S_\ell^{-1}\circ\Phi_\theta\circ S_\ell$};
\draw[densely dashed, draw=elippsblue!60] (f.south) -- (sym.north);

\node[lab, above=1.5mm of x, text=elippsaccent] {unknown};
\node[lab, above=1.5mm of Ax, text=elippsaccent] {never observed};
\node[lab, above=1.5mm of xhat, align=center, text=elippsblue]
     {reconstruction\\[-2pt] at test time};
\node[lab, above=1.5mm of loss, anchor=south east, xshift=-1.5mm,
      text=elippsaccent] {training loss};
      
\end{tikzpicture}
\caption{Overview of ELIPPS. Only the undersampled measurement $y_I = P_IAx+\noise_I$ and  $y_S = P_SAx + \noise_S$ on a fixed supervision set $S$  are available during training. Neither the image $x$ nor the
full measurement $Ax$ is ever observed. The reconstruction network is the group-averaged
image-space network $f_\theta$, applied to an initial reconstruction $x^{0}=B_I(y_I)$, and the
loss compares its re-projection only on $S$. Theorem~\ref{thm:main} and
Corollary~\ref{cor:elipps} show that minimizing this masked loss over the resulting equivariant
class recovers $\E[AX\mid Y_I]$, and Corollary~\ref{cor:image} that $\hat x$ agrees with the minimum mean
squared error estimate of $x$ given $y_I$ up to the null space of $A$.}
\label{fig:method}
\end{figure}

\section{Incomplete Self-Supervised Learning }

 In this section we introduce and analyze \emph{Exact Learning for Inverse Problems from Partial
Self-supervision} (ELIPPS) for solving \eqref{eq:ip} using partial self-supervision
\eqref{eq:partial}. We adopt a probabilistic setting on a common probability space
$(\Omega,\mathcal{F},\mathbb{P})$: the unknown image $x$ is the realization of a random vector (RV)
$X$ with unknown distribution; the observed subsampled measurements are $Y_I = A_I X + \noise_I$
and the supervision measurement is $Y_S = A_S X + \noise_S$. All expectations are taken with
respect to the joint distribution induced by \eqref{eq:ip} and \eqref{eq:partial}. For RVs $Y, Z$ we write $Y \stackrel{d}{=} Z$ for equality in distribution and $Y = Z$ for
almost sure equality (unless indicated otherwise, \emph{almost surely} always refers to
$\mathbb{P}$).
Throughout, we write $\mathbb{R}^I$ for the coordinate space indexed by $I$ and identify
$\mathbb{R}^I \simeq \mathbb{R}^I \times \{0\} \subseteq \mathbb{R}^M$, so that the orthogonal
projector $P_I$ acts as a binary masking operator on $\mathbb{R}^M$, zeroing out the coordinates
outside $I$; analogously for $S$.

\subsection{Preliminaries}
\label{sec:prelim}

Let $T:\mathbb{R}^M \to \mathbb{R}^M$ be an invertible linear map. The first key ingredient of ELIPPS is invariance of the underlying signal class.

\begin{definition}[$T$-invariant RV]
Let $T:\mathbb{R}^M \to \mathbb{R}^M$ be an invertible linear map. A RV $Y$ on $\mathbb{R}^M$ is said to be invariant with respect to $T$ (or $T$-invariant, in distribution) if
$Y \stackrel{d}{=} T(Y)$.\end{definition}

Note that $Y \stackrel{d}{=} T(Y)$ is much weaker than $Y = T(Y)$. Consider a
distribution of images that is invariant under a rotation $T$ by $90^\circ$, for instance because
every orientation of every image occurs with the same probability. Then $Y \stackrel{d}{=} T(Y)$,
whereas $T(Y) = Y$ would require every single image to be unchanged by the rotation, which holds
only for rotationally symmetric images. We write $\pi_Y$ for the law of $Y$, i.e.\
$\pi_Y(E) = \mathbb{P}(Y \in E)$ for Borel sets $E \subseteq \mathbb{R}^M$. In this notation the
$T$-invariance of $Y$ reads $\pi_{T(Y)} = \pi_Y$, or equivalently
$\mathbb{P}(Y \in E) = \mathbb{P}(T(Y) \in E)$ for all $E$ Borel.

\begin{definition}[Joint $T$-invariance]
Let $Y, Z$ be RVs on $\mathbb{R}^M$.  The pair $(Z,Y)$ is said to be jointly $T$-invariant in distribution if $(Z,Y) \stackrel{d}{=} (T(Z), T(Y))$.
\end{definition}

Besides invariance of the  of RVs, we also use invariance of the sampling set $I$.

\begin{definition}[$T$-invariant subsampling]
We say that the sampling set $I \subseteq \{1,\dots,M\}$ is \emph{$T$-invariant} if
$P_I T = T P_I$.
\end{definition}

\begin{remark} \label{rem:Iinv}
The commutation $P_I T = T P_I$ holds precisely when $T$ is block-diagonal with respect to
$\mathbb{R}^M = \mathbb{R}^I \times \mathbb{R}^J$ with $J = \{1,\dots,M\} \setminus I$; that is,
when $T$ maps the observed coordinates among themselves and likewise the unobserved ones. For a
permutation $T$ this means that $I$ is a union of cycles of $T$, and then so is $J$ automatically;
this is how $P_I T = T P_I$ is verified in Section~\ref{sec:experiments}. For the main results it
suffices to assume joint invariance directly, without the commutativity. We state the latter
explicitly because it is what the sufficient condition of Lemma~\ref{lem:assump-to-joint} and the
sample-wise identity of Proposition~\ref{prop:orbit} use.
\end{remark}

\begin{lemma}[Sufficient conditions for joint invariance]\label{lem:assump-to-joint}
Let $Y$ be a RV on $\mathbb{R}^M$, let $I\subseteq\{1,\dots,M\}$, and let
$Y_I = P_I Y + \noise_I$ with $\noise_I$ independent of $Y$. Assume that $Y$ and $\noise_I$ are
$T$-invariant and that $I$ is a $T$-invariant sampling set. Then the pair $(Y_I,Y)$ is jointly
$T$-invariant.
\end{lemma}

\begin{proof}
Since $\noise_I$ is independent of $Y$ and $T$ is deterministic, $T\noise_I$ is independent of
$TY$. Both pairs $(Y,\noise_I)$ and $(TY, T\noise_I)$ therefore have product laws, and their
marginals agree because $Y \stackrel{d}{=} TY$ and $\noise_I \stackrel{d}{=} T\noise_I$; hence
$(Y,\noise_I) \stackrel{d}{=} (TY, T\noise_I)$. Applying the measurable map
$(y,e) \mapsto (P_I y + e,\, y)$ to both sides gives
$(Y_I, Y) \stackrel{d}{=} \bigl(P_I(TY) + T\noise_I,\, TY\bigr)$. Since $I$ is $T$-invariant and
$T$ is linear, $P_I(TY) + T\noise_I = T(P_I Y + \noise_I) = T Y_I$, which is the claim.
\end{proof}

\begin{definition}[Equivariance]\label{def:equivariance}
A function $f \colon \mathbb{R}^M \to \mathbb{R}^M$ is called $T$-equivariant, for an invertible
linear map $T$, if $f \circ T = T \circ f$. For a finite family $\T = (T_\ell)_{\ell=1}^L$ of
invertible linear maps, $f$ is called $\T$-equivariant if it is $T_\ell$-equivariant for every
$\ell = 1, \dots, L$. The set of all measurable $\T$-equivariant functions
$f \colon \mathbb{R}^M \to \mathbb{R}^M$ is denoted by $\Equi(\mathbb{R}^M;\T)$.
\end{definition}

The conditional expectation $\mathbb{E}[Y \mid Y_I]$ depends on $Y_I$ only through the
$\sigma$-algebra $\sigma(Y_I) = Y_I^{-1}(\mathcal{B})$ it generates. It is a random variable of
the form $h(Y_I)$ for a Borel function $h$ on $\mathbb{R}^M$, named Borel version of the \emph{regression function}, written as $h(y) = \mathbb{E}[Y \mid Y_I = y]$, and determined only up to
$\pi_{Y_I}$-null sets. Notice that the equivariance induced by joint invariance is a property of $h$ rather
than of $\mathbb{E}[Y \mid Y_I]$: since $T$ is invertible, $\sigma(T Y_I) = \sigma(Y_I)$, so
nothing is gained at the level of the conditional expectation. 

\begin{lemma}[Joint invariance implies equivariance of the regression function]\label{lem:joint-to-ce}
Let $(Y_I,Y)$ be jointly $T$-invariant and square-integrable, and let
$h \colon \mathbb{R}^M \to \mathbb{R}^M$ be a Borel version of the regression function $h(y) = \mathbb{E}[Y \mid Y_I = y]$. Then $h(T y) = T\,h(y)$  for $\pi_{Y_I}\text{-almost every } y \in \mathbb{R}^M$.
\end{lemma}

\begin{proof}
Since $T$ is deterministic and linear, a version of the regression function of $T Y$ given $T Y_I$
is $y \mapsto T\,h(T^{-1} y)$; indeed, for every bounded Borel $\varphi$,
\[
\mathbb{E}\bigl[T Y\,\varphi(T Y_I)\bigr]
= T\,\mathbb{E}\bigl[Y\,\varphi(T Y_I)\bigr]
= T\,\mathbb{E}\bigl[h(Y_I)\,\varphi(T Y_I)\bigr]
= \mathbb{E}\bigl[T h(T^{-1}(T Y_I))\,\varphi(T Y_I)\bigr].
\]
By joint $T$-invariance, $(T Y_I, T Y) \stackrel{d}{=} (Y_I, Y)$, so the two pairs have the same
regression function up to a $\pi_{Y_I}$-null set; in particular $\pi_{T Y_I} = \pi_{Y_I}$ and
$
h(y) = T\,h(T^{-1} y) $ for $\pi_{Y_I}$-almost every $y \in \mathbb{R}^M$
which is the assertion after replacing $y$ by $T y$.
\end{proof}

Lemma~\ref{lem:joint-to-ce} determines $h$ only up to a $\pi_{Y_I}$-null set, whereas
Theorem~\ref{thm:main} below minimizes over the class of \emph{everywhere} defined equivariant maps. The
gap is closed by symmetrization, which also explains why $\T$ is required to be a group
(see Condition~\ref{cond:group} below).

\begin{lemma}[Globally equivariant version]\label{lem:global-version}
Let $\T = (T_\ell)_{\ell=1}^L$ be a finite group of invertible linear maps, let $(Y_I,Y)$ be jointly $T_\ell$-invariant for every $\ell$, and let $h$ be a Borel
version of the regression function of $Y$ given $Y_I$. Then $\tilde h := \frac{1}{L}\sum_{\ell=1}^{L} T_\ell^{-1}\circ h \circ T_\ell $ is $\T$-equivariant and satisfies $\tilde h = h$
$\pi_{Y_I}$-almost everywhere; in particular $\tilde h(Y_I) = \mathbb{E}[Y\mid Y_I]$ almost surely
and $\tilde h \in \Equi(\mathbb{R}^M;\T)$.
\end{lemma}

\begin{proof}
Fix $k$. Since $\T$ is a group, $\ell \mapsto T_\ell T_k$ permutes $\T$, hence $
T_k^{-1}\circ\tilde h\circ T_k
= \frac{1}{L}\sum_{\ell=1}^{L} (T_\ell T_k)^{-1}\circ h\circ (T_\ell T_k)
= \tilde h$, so $\tilde h \circ T_k = T_k \circ \tilde h$ everywhere. By Lemma~\ref{lem:joint-to-ce} each
summand $T_\ell^{-1} h(T_\ell\, \cdot\,)$ agrees with $h$ outside a $\pi_{Y_I}$-null set, using that
$\pi_{Y_I}$ is $T_\ell$-invariant; a finite union of null sets is null, so $\tilde h = h$
$\pi_{Y_I}$-a.e.
\end{proof}

By Lemmas~\ref{lem:joint-to-ce} and~\ref{lem:global-version}, the optimal predictor
$y \mapsto \mathbb{E}[Y \mid Y_I = y]$ admits a $\T$-equivariant version; that is, applying $T$
before or after denoising and upsampling yields the same result.

\subsection{Measurement Self-supervision}
\label{sec:selfsup}

We target the image restoration problem~\eqref{eq:ip} of reconstructing a signal or image $X \in \mathbb{R}^N$ from incomplete, noisy, indirect observations. We consider a random setting with deterministic design and partial self-supervision~\eqref{eq:partial}. The method and main results are derived under the following assumptions.

\begin{condition}\label{cond:main}
Assume the following hold:
\begin{enumerate}[label=(A\arabic*)]
    \item $A \colon \mathbb{R}^N \to \mathbb{R}^M$ is a fixed potentially nonlinear forward map. 
    \item $I, S \subseteq \{1, \dots, M\}$ are fixed subsampling (input) and supervision index sets.
    \item\label{cond:model} The data are generated according to the partial self-supervision model \eqref{eq:partial}.

\item\label{cond:noise} $X$, $\noise_I$ and $\noise_S$ are square-integrable, and  $\mathbb{E}[\noise_S \mid X, \noise_I] = 0$.

    \item\label{cond:group} $\T = (T_\ell)_{\ell=1}^L$ is a finite group of invertible linear operators on $\mathbb{R}^{M}$, of order $L$. 
    \item\label{cond:jointinv} For every $\ell$, the pair $(Y_I,A X)$ is jointly $T_\ell$-invariant.
    \item\label{cond:frame} There exists $c>0$ such that for all $v \in \mathbb{R}^{M} \colon
    \sum_{\ell=1}^L \|P_S\, T_\ell v\|_2^2 = c\,\|v\|_2^2$.
\end{enumerate}
\end{condition}

\begin{remark}[On the noise assumption \ref{cond:noise}]\label{rem:A4}
The usual assumption in this context, that $\noise_S$ be zero-mean and independent of
$(X,\noise_I)$, implies \ref{cond:noise}, but not conversely.  In particular, \ref{cond:noise} requires the
noise neither to be additive nor signal-independent nor identically distributed across
coordinates, and it imposes no assumption whatsoever on the joint law of $X$ and $\noise_I$.
It is satisfied whenever the supervision measurement is an unbiased acquisition that is
conditionally independent of the input acquisition given the signal: if
$\mathbb{E}[Y_S \mid X] = A_S X$ and $Y_S$ is conditionally independent of $\noise_I$ given
$X$, then the latter gives $\mathbb{E}[Y_S \mid X,\noise_I] = \mathbb{E}[Y_S \mid X]$, whence
$\mathbb{E}[\noise_S \mid X,\noise_I] = \mathbb{E}[Y_S - A_S X \mid X,\noise_I]
= \mathbb{E}[Y_S \mid X] - A_S X = 0$.
This covers signal-dependent noise, in particular Poisson statistics, and requires no Gaussian
approximation. It does require the supervision measurement to be unbiased, which is not
automatic: if the sinogram is formed from photon counts by a log transform, a systematic bias
enters. Remark~\ref{rem:bias} below shows that such a bias propagates in a controlled way.
\end{remark}

Note that Condition~\ref{cond:noise} implies that the supervision measurement must not
reuse the noise of the input measurement. Physically this leaves two options: either two
conditionally independent acquisitions of the same object, in the spirit of Noise2Noise
\cite{Lehtinen2018}, or a single acquisition split into two independent parts, achieved for
photon counts by Poisson thinning at unchanged total dose. A disjoint design
$I \cap S = \emptyset$ would also suffice, but \ref{cond:frame} forces $S$ to meet every orbit
while $P_I T_\ell = T_\ell P_I$ forces $I$ to be a union of orbits, so the two sets must
overlap; see Section~\ref{sec:experiments}.
 
Throughout, every minimization over a class of maps is understood to range over those $g$ in the class for which $g(Y_I)$ is square-integrable. Identities between minimizers are understood in the usual
$L^2$ sense: the $\argmin$ is the set of minimizers of the stated risk, all of whose elements
agree $\pi_{Y_I}$-almost everywhere, and an equation of the form
$\mathbb{E}[AX\mid Y_I] = \argmin(\cdot)$ asserts that every minimizer $g$ satisfies
$g(Y_I) = \mathbb{E}[AX\mid Y_I]$ almost surely. Whether the minimum is attained is addressed by
Lemma~\ref{lem:global-version} and, for the constrained classes of Sections~\ref{sec:recon}
and~\ref{sec:freedom}, by Corollary~\ref{cor:equivariant}.

\begin{theorem}[Exact learning from partial self-supervision]\label{thm:main}
Suppose Condition~\ref{cond:main} holds. Then, over the equivariant function class $\Equi(\mathbb{R}^M;\T)$, minimizing the partially self-supervised noisy-target risk recovers the clean conditional expectation:
\[
\mathbb{E}[A X \mid Y_I]
=
\argmin_{g \in \Equi(\mathbb{R}^M;\T)} \;
\Esq{P_S \, g(Y_I) - Y_S}.
\]
\end{theorem}

\begin{proof}
Let $Z = \mathbb{E}[A X \mid Y_I]$; by Lemmas~\ref{lem:joint-to-ce} and~\ref{lem:global-version} (applied with $Y = AX$) it is attained by a map in $\Equi(\mathbb{R}^M;\T)$. For any square-integrable $g:\mathbb{R}^M\to\mathbb{R}^M$, expand the noisy-target risk:
\begin{equation*}
\Esq{P_S g(Y_I)-Y_S}
= \Esq{P_S g(Y_I)-P_S A X} + \Esq{\noise_S}
- 2\,\mathbb{E}\!\left[\langle P_S(g(Y_I)-A X),\noise_S\rangle\right].
\end{equation*}
Since $g(Y_I) - A X$ is $\sigma(X,\noise_I)$-measurable, conditioning on $(X,\noise_I)$ and using \ref{cond:noise} shows that the cross term vanishes. Since $\Esq{\noise_S}$ does not depend on $g$, minimizing the masked noisy risk is equivalent to minimizing the masked clean risk $\Esq{P_S g(Y_I)-P_S A X}$.

Let $g$ be $\T$-equivariant and define $R = g(Y_I) - A X$. By equivariance, $g(T_\ell Y_I) - T_\ell A X = T_\ell R$; by joint invariance (Condition~\ref{cond:jointinv}), $\Esq{P_S R} = \Esq{P_S T_\ell R}$ for all $\ell$. Averaging over $\ell$ and using Condition~\ref{cond:frame} yields
\begin{equation*}
\Esq{P_S g(Y_I) - P_S A X}
= \frac{1}{L} \, \mathbb{E}\!\left[ \sum_{\ell=1}^L \|P_S T_\ell R\|_2^2 \right]
= \frac{c}{L}\, \Esq{R}
= \frac{c}{L}\, \Esq{g(Y_I) - A X}.
\end{equation*}
Therefore, within $\Equi(\mathbb{R}^M;\T)$, the masked clean risk is strictly proportional to the full clean risk, and both are minimized by $Z = \mathbb{E}[A X \mid Y_I]$.
\end{proof}

\begin{remark}[Biased supervision]\label{rem:bias}
Suppose that instead of \ref{cond:noise} the supervision noise satisfies
$\mathbb{E}[\noise_S \mid X, \noise_I] = P_S\,\beta(X)$ for a measurable bias map 
$\beta \colon \R^N \to \R^M$, and that the pair $(Y_I, AX + \beta(X))$ is jointly
$T_\ell$-invariant for every $\ell$ (which holds, for instance, when $\beta$ acts
coordinatewise on $Ax$ and the $T_\ell$ are coordinate permutations). Writing
$\noise_S = P_S\beta(X) + \eta$ with $\mathbb{E}[\eta\mid X,\noise_I] = 0$ and repeating the
proof of Theorem~\ref{thm:main} with $AX$ replaced by $AX + \beta(X)$ gives
\[
\argmin_{g\in\Equi(\mathbb{R}^M;\T)} \Esq{P_S g(Y_I) - Y_S}
= \mathbb{E}\bigl[AX + \beta(X) \,\big|\, Y_I\bigr] .
\]
The minimizer of the partially self-supervised risk is therefore again a conditional
expectation, now of the biased quantity $AX + \beta(X)$ instead of $AX$: the learned predictor
inherits the bias of the supervision measurement and nothing else. Its deviation from
$\mathbb{E}[AX\mid Y_I]$ is $\mathbb{E}[\beta(X)\mid Y_I]$, whose $L^2$ norm is at most
$\|\beta(X)\|_{L^2}$ because conditional expectation is an $L^2$ contraction. The latter can be
estimated from a concrete acquisition model; see Section~\ref{sec:experiments}.
\end{remark}

The proof gives more than the statement of Theorem~\ref{thm:main}: on the whole equivariant class the two risks differ only by a positive factor and an additive constant, which yields an exactness statement independent of any modelling assumption.

\begin{corollary}[Exactness within arbitrary equivariant subclasses]\label{cor:subclass}
Suppose Condition~\ref{cond:main} holds and let $\mathcal{C} \subseteq \Equi(\mathbb{R}^M;\T)$ be an arbitrary nonempty class of
     square-integrable equivariant maps. Then
\[
\argmin_{g \in \mathcal{C}} \Esq{P_S\, g(Y_I) - Y_S}
=
\argmin_{g \in \mathcal{C}} \Esq{g(Y_I) - \mathbb{E}[AX \mid Y_I]} ,
\]
i.e.\ training on partially self-supervised, noisy data over $\mathcal{C}$ returns the best
$L^2$-approximation of $\mathbb{E}[AX\mid Y_I]$ available in $\mathcal{C}$.
\end{corollary}

\begin{proof}
Write $Z = \mathbb{E}[AX \mid Y_I]$. By the two displays in the proof of Theorem~\ref{thm:main},
every $g \in \Equi(\mathbb{R}^M;\T)$ satisfies
$ \Esq{P_S g(Y_I) - Y_S}
= \frac{c}{L}\,\Esq{g(Y_I) - AX} + \Esq{\noise_S}$. Since $g(Y_I)$ is $\sigma(Y_I)$-measurable and $\mathbb{E}[AX - Z \mid Y_I] = 0$, the Pythagoras
identity for conditional expectations gives $\Esq{g(Y_I) - AX}
= \Esq{g(Y_I) - Z} + \Esq{AX - Z}$, and the last term does not depend on $g$. The two objectives therefore differ by the positive
factor $c/L$ and an additive constant, which do not change the set of minimizers.
\end{proof}

Corollary~\ref{cor:subclass} makes ELIPPS robust with respect to modelling choices. Whatever the hypothesis class, minimizing the partially self-supervised risk over an equivariant class is therefore \emph{always} equivalent to minimizing the full self-supervised risk over the same class. Assumptions beyond Condition~\ref{cond:main}, in particular those of Section~\ref{sec:freedom}, are never needed for the correctness of the training objective; they only decide whether $\mathbb{E}[AX\mid Y_I]$ lies in the class. 

Theorem~\ref{thm:main} is a statement about expectations. When the training set is closed under the group action, a purely deterministic statement is available as well: on the orbit of a single
sample, the masked loss and the full unmasked loss coincide.

\begin{proposition}[Orbit identity]\label{prop:orbit}
Let $\T$ satisfy \ref{cond:frame}, let $P_I T_\ell = T_\ell P_I$ for every
$\ell$, and let $g \in \Equi(\mathbb{R}^M;\T)$. For a measurement vector
$y \in \mathbb{R}^M$ put $y_I = P_I y$, and consider the orbit of $y$ under $\T$ with the associated
input and supervision data $y_I^{(\ell)} = P_I T_\ell y$ and $y_S^{(\ell)} = P_S T_\ell y$. Then
\[
\sum_{\ell=1}^{L}\bigl\|P_S\, g\bigl(y_I^{(\ell)}\bigr) - y_S^{(\ell)}\bigr\|_2^2
\;=\; c\,\bigl\|g(y_I) - y\bigr\|_2^2 .
\]
\end{proposition}

\begin{proof}
Since $P_I$ commutes with $T_\ell$ we have $y_I^{(\ell)} = T_\ell P_I y = T_\ell y_I$, and
equivariance of $g$ gives $g(y_I^{(\ell)}) = T_\ell\, g(y_I)$. Hence
$P_S g(y_I^{(\ell)}) - y_S^{(\ell)} = P_S T_\ell v$ with $v = g(y_I) - y$, and summing over $\ell$
and using \ref{cond:frame} yields $\sum_\ell \|P_S T_\ell v\|_2^2 = c\|v\|_2^2$.
\end{proof}

The identity in Proposition~\ref{prop:orbit} is exact and finite: it involves no probability distribution, no invariance
assumption on the data and no expectation. If the training set is closed under the group
action, the empirical masked risk over the orbit-closed set therefore equals $c$ times the
empirical \emph{full} risk over the original samples, for every equivariant candidate.
Enlarging the training set by the group is thus not a way of obtaining more data but the
finite realization of the coverage condition~\ref{cond:frame}, which the next proposition
turns into a characterization. Theorem~\ref{thm:main} is the expectation version of the same
mechanism. There the noise contributes a cross term that vanishes only in the mean, which is
precisely where \ref{cond:noise} and the joint invariance~\ref{cond:jointinv} enter.

\begin{proposition}[Characterization of the coverage condition]\label{prop:characterization}
Let $\T = (T_\ell)_{\ell=1}^L$ be a finite set of invertible linear maps, not necessarily a
group, and let $S \subseteq \{1,\dots,M\}$. Then \ref{cond:frame} holds, that is
$\sum_{\ell} \norm{P_S T_\ell v}_2^2 = c\,\norm{v}_2^2$ for every $v \in \mathbb{R}^M$, if and
only if
\begin{equation}\label{eq:frameop}
\sum_{\ell=1}^{L} T_\ell^{\!\top} P_S\, T_\ell = c\,\id .
\end{equation}
\end{proposition}

\begin{proof}
For every $v$ we have
$\sum_\ell \norm{P_S T_\ell v}_2^2
= \sum_\ell \langle T_\ell^{\!\top} P_S^{\!\top} P_S T_\ell\, v, v\rangle
= \langle \sum_\ell T_\ell^{\!\top} P_S T_\ell\, v, v\rangle$,
because $P_S^{\!\top} P_S = P_S$ for an orthogonal projection. A symmetric matrix is determined
by its quadratic form, so this equals $c\norm{v}_2^2$ for all $v$ if and only if
\eqref{eq:frameop} holds.
\end{proof}

In the language of frame theory, \eqref{eq:frameop} states that the rows of the maps
$(P_S T_\ell)_{\ell=1}^L$ form a tight frame of $\mathbb{R}^M$ with frame constant $c$. For
orthogonal $T_\ell$, taking traces gives $\abs{S} = c\,M/L$, which is the quantitative content
of the condition: the supervision set must contain a fraction $c/L$ of all measurements, so the achievable compression is bounded by $L$. Under Condition~\ref{cond:main}, where $\T$ is a group, $L$ is its order; for the dihedral group $D_4$ used in Section~\ref{sec:experiments} this is a factor of eight.
Whether $c = 1$ is attainable depends on the action. If the $T_\ell$ are permutations and a
coordinate $i$ is fixed by $k$ of them, then every coordinate that $S$ reaches in the same orbit
contributes $k$ to the multiplicity of $i$, so tightness forces $c \geq k$ and hence
$\abs{S} \geq k\,M/L$. Only a free action, in which no coordinate is fixed by a nontrivial group
element, admits $c = 1$ and $\abs{S} = M/L$. On a discrete grid the reflection axes of $D_4$
produce fixed coordinates; the weighted formulation of Remark~\ref{rem:weights} removes the
obstruction by counting them with the reciprocal of their multiplicity.
Conversely, the condition tells us how to choose $S$: as a fundamental domain of the group action,
which is a design decision rather than something found in the data.

\begin{remark}[Weighted supervision masks]\label{rem:weights}
Let $W \in \mathbb{R}^{M\times M}$ be a diagonal matrix with nonnegative entries supported in $S$,
and replace \ref{cond:frame} by the weighted coverage condition
\begin{equation}\label{eq:frameW}
\sum_{\ell=1}^L \bigl\|W^{1/2} P_S\, T_\ell v\bigr\|_2^2 = c\,\|v\|_2^2
\qquad \text{for all } v \in \mathbb{R}^M .
\tag{A7$_W$}
\end{equation}
Then Theorem~\ref{thm:main} and Corollary~\ref{cor:subclass} remain valid verbatim for the
weighted risk $\Esq{W^{1/2}(P_S g(Y_I) - Y_S)}$. Indeed, the only property of $P_S$
used in the proofs is that it is a fixed deterministic linear map, so that
$R \stackrel{d}{=} T_\ell R$ implies $\Esq{W^{1/2} P_S R} = \Esq{W^{1/2} P_S T_\ell R}$.
The additional freedom is useful on discrete grids, where the group action typically has fixed
points and an unweighted mask can be a fundamental domain only up to those fixed points; see
Section~\ref{sec:experiments}.
\end{remark}

Condition \ref{cond:frame} requires the transformed supervision sets to cover the measurement
domain exactly. On a discrete grid this can fail by a small amount.  The following stability estimate shows that the conclusion of Theorem~\ref{thm:main} survives on the part of the measurement domain that is covered at all, no matter how unevenly.

\begin{theorem}[Stability under approximate coverage]\label{thm:stability}
Suppose Condition~\ref{cond:main} holds with \ref{cond:frame} replaced by the following
one-sided condition: there exist a constant $c_1 > 0$ and an orthogonal coordinate projection $Q$
on $\mathbb{R}^M$ such that
\begin{equation}\label{eq:frameapprox}
c_1 \|Q v\|_2^2 \;\leq\; \sum_{\ell=1}^L \|P_S\, T_\ell v\|_2^2
\qquad \text{for all } v \in \mathbb{R}^M .
\end{equation}
Let $Z = \mathbb{E}[AX\mid Y_I]$ and let $\hat g \in \Equi(\mathbb{R}^M;\T)$ be a
$\delta$-minimizer of the partially self-supervised risk, i.e.\
$\Esq{P_S \hat g(Y_I) - Y_S} \leq \inf_{g} \Esq{P_S g(Y_I) - Y_S} + \delta$
for some $\delta \geq 0$. Then
\[
\Esq{Q\bigl(\hat g(Y_I) - Z\bigr)} \;\leq\; \frac{L}{c_1}\,\delta .
\]
In particular, for $\delta = 0$ the conditional expectation $Z$ is identified exactly on the range of $Q$. 
\end{theorem}

\begin{proof}
Write $M_S := \sum_{\ell=1}^L T_\ell^{\!\top} P_S\, T_\ell$, so that
$\sum_\ell \|P_S T_\ell v\|_2^2 = \langle M_S v, v\rangle$ and \eqref{eq:frameapprox} reads
$c_1\|Qv\|_2^2 \leq \langle M_S v, v\rangle$. For $g \in \Equi(\mathbb{R}^M;\T)$ put
$R_g = g(Y_I) - AX$ and $D_g = g(Y_I) - Z$. As in the proof of Theorem~\ref{thm:main} the
noisy-target risk equals $\Esq{P_S R_g} + \Esq{\noise_S}$ and $R_g \stackrel{d}{=} T_\ell R_g$ for
every $\ell$, so
\[
L\,\Esq{P_S R_g} \;=\; \mathbb{E}\Bigl[\textstyle\sum_\ell \|P_S T_\ell R_g\|_2^2\Bigr]
\;=\; \mathbb{E}\bigl[\langle M_S R_g, R_g\rangle\bigr].
\]
Decompose $R_g = D_g + (Z - AX)$. Since $D_g$ is $\sigma(Y_I)$-measurable, so is $M_S D_g$, and
$\mathbb{E}[Z - AX \mid Y_I] = 0$, so the cross term vanishes and
$
\mathbb{E}\bigl[\langle M_S R_g, R_g\rangle\bigr]
= \mathbb{E}\bigl[\langle M_S D_g, D_g\rangle\bigr]
+ \mathbb{E}\bigl[\langle M_S (Z - AX), Z - AX\rangle\bigr]$, 
where the second term does not depend on $g$. The risk is therefore, up to an additive constant, a
quadratic form in $D_g$ alone. By Lemma~\ref{lem:global-version}, applied with $Y = AX$, the value
$Z$ is attained within $\Equi(\mathbb{R}^M;\T)$, and there $D_g = 0$, so $\delta$-minimality of
$\hat g$ gives $\mathbb{E}[\langle M_S D_{\hat g}, D_{\hat g}\rangle] \leq L\delta$. The lower bound
in \eqref{eq:frameapprox} now yields
$c_1\,\Esq{Q D_{\hat g}} \leq \mathbb{E}[\langle M_S D_{\hat g}, D_{\hat g}\rangle] \leq L\delta$,
which is the assertion.
\end{proof}

The proof applies \ref{cond:frame} only to the residual $R = g(Y_I) - AX$, so requiring it on $\range(A)$ suffices for the factorized class of Section~\ref{sec:recon}. This is a genuine weakening whenever $A$ has a nontrivial cokernel, but not for the discretization of Section~\ref{sec:experiments}, where $\range(A) = \R^M$ is expected.  The estimate depends on the multiplicities only through the lower bound $c_1$: uneven overlap of the images $T_\ell(S)$ reweights the risk but does not move its minimizer, and the conditional expectation is still identified exactly wherever $\ref{cond:frame}$ fails only in this way. What cannot be repaired is a coordinate reached by no image at all, which is why $Q$ appears in the statement.

\subsection{Self-supervised image reconstruction}
\label{sec:recon}

We realize the optimal predictor of Theorem~\ref{thm:main} as $g_\theta = A \circ f_\theta \circ B$ with an image-space network $f_\theta \colon \R^N \to \R^N$, a fixed initial reconstruction $B \colon \R^M \to \R^N$ such as filtered backprojection, and the forward operator. $\T$-equivariance of $g_\theta$ is enforced through intertwining relations between the measurement-space transforms $\T$ and the image-space transforms $\TS$, together with symmetric averaging of $f_\theta$ in image space.

\begin{condition} \label{cond:ST}
Let $\T = (T_\ell)_{\ell=1}^L \in \mathrm{GL}(\mathbb{R}^M)^L$ and $\TS = (S_\ell)_{\ell=1}^L \in \mathrm{GL}(\mathbb{R}^N)^L$ be finite families of invertible linear operators satisfying:
\begin{enumerate}
    \item (Intertwining with $A$) For all $\ell \in \{1, \dots L\} $: 
    $T_\ell \circ A \;=\; A \circ S_\ell$
    \item (Intertwining of $B$) For all $\ell \in \{1, \dots L\} $: 
    $B \circ T_\ell \;=\; S_\ell \circ B$
    \item (Group structure) $\TS = (S_\ell)_{\ell=1}^L$ is a finite group of order $L$.
\end{enumerate}
\end{condition}

Condition~\ref{cond:ST} is naturally satisfied in tomographic problems with known symmetries  such as the Radon transform, where rotations in image space correspond to translations of the sinogram.

\begin{remark}[Two opposite uses of a symmetry group]\label{rem:opposite}
Methods that learn from $y_I$ alone use the invariance of the signal distribution to read the same
data as a measurement of a transformed signal through a transformed operator. For that mechanism to supply information, the forward operator must \emph{not} commute with the group, since otherwise all virtual operators share the null space of $A$ \cite{tachella2023sensing,tachella2026selfsup}. Condition~\ref{cond:ST}
asks for the intertwining $T_\ell A = A S_\ell$ to hold exactly, so we are in the regime that
excludes it. There is no contradiction: here the supervision comes from $y_S$, and the group is not
asked to produce new information but to \emph{transport the supervision that exists on $S$ across the
measurement domain}, which is what \ref{cond:frame} formalizes. Consistently, we make no claim about
the null space; Corollary~\ref{cor:image} identifies the reconstruction only modulo $\ker(A)$.
\end{remark}

\begin{theorem}
\label{thm:Xnet}
For given linear $B \colon \mathbb{R}^M \to \mathbb{R}^N$, let  $\T=(T_\ell)_{\ell=1}^L$ and $\TS=(S_\ell)_{\ell=1}^L$ satisfy Condition~\ref{cond:ST}. For any measurable $\Phi \colon \mathbb{R}^N \to \mathbb{R}^N$, define     \begin{align}
\label{eq:symmetrized_network}
f &\colon \mathbb{R}^N \to \mathbb{R}^N, \qquad
f \coloneqq \frac{1}{L}\sum_{\ell=1}^L S_\ell^{-1} \circ \Phi \circ S_\ell ,
\\
\label{eq:symmetrized_network_g}
g &\colon \mathbb{R}^M \to \mathbb{R}^M, \qquad
g \coloneqq A \circ f \circ B \,.
\end{align}
Then $f$, and $g$ are $\TS$-equivariant and $\T$-equivariant, respectively.
\end{theorem}

\begin{proof}
Fix $k \in \{ 1, \dots, L\}$. Using the intertwining relations $T_k A = A S_k$ and $B T_k = S_k B$ (Condition~\ref{cond:ST}), we have
\begin{align*}
T_k \circ g
& =
T_k \circ (A \circ f \circ B)
=
A \circ (S_k \circ f ) \circ B,
\\
g \circ T_k
&=
(A \circ f \circ B ) \circ T_k
=
A \circ ( f \circ S_k ) \circ B.
\end{align*}
Hence, it suffices to show that $f$ commutes with every $S_k$. Since $\TS$ is a group,
$S_k^{-1} \in \TS$ and the assignment $\ell \mapsto m(\ell)$ defined by
$S_{m(\ell)} = S_\ell \circ S_k^{-1}$ is a bijection of $\{1,\dots,L\}$. Using first the identity
$S_k \circ S_\ell^{-1}\circ\Phi\circ S_\ell
= \bigl((S_\ell \circ S_k^{-1})^{-1}\circ\Phi\circ(S_\ell \circ S_k^{-1})\bigr)\circ S_k$,
which holds term by term, and then this reindexing, we obtain
\begin{multline*}
S_k \circ f
=
\frac{1}{L}\sum_{\ell=1}^L
   S_k \circ  ( S_\ell^{-1}  \circ \Phi \circ S_\ell )
\\=
\frac{1}{L}\sum_{\ell=1}^L
\bigl( (S_\ell \circ S_k^{-1})^{-1} \circ \Phi \circ (S_\ell \circ S_k^{-1}) \bigr) \circ S_k
=
\frac{1}{L}\sum_{m=1}^L
(   S_m^{-1}  \circ  \Phi \circ S_m ) \circ S_k
=
 f \circ S_k \,,
\end{multline*}
which concludes the proof.
\end{proof}

The averaging in \eqref{eq:symmetrized_network} is the classical Reynolds operator, and its use to
render a network equivariant is standard; it appears in the learning literature as group or frame
averaging \cite{Cohen2016,Puny2022,Sannai2024} and, for imaging, in test-time symmetrization of
plug-and-play denoisers \cite{Terris2024}. What is specific here is not the averaging itself but its
combination with the forward operator and the initial reconstruction, together with the statement
that the resulting restriction of the hypothesis class costs nothing
(Theorem~\ref{thm:arch} and Corollary~\ref{cor:equivariant}).  No architectural constraints on $\Phi$ are required; equivariance is enforced purely by the symmetrization. Thus, $\Phi$ can be any CNN operating in image space. In practice, $B$ can be chosen as a standard analytic or approximate inverse (for example FBP for the Radon transform), which typically intertwines with rotations and reflections.

\subsection{Architectural freedom}
\label{sec:freedom}

Realizing $g_\theta = A \circ f_\theta \circ B$ guarantees the equivariance of the network by
construction. In this subsection we show that this comes at no cost: the restriction to the factorized class leaves the solution of the minimization problem unchanged.

\begin{definition}[Right-inverse and least-squares right-inverse]\label{def:lsri}
    Let $A\colon\R^N\to\R^M$ be a linear operator and $P_{\range(A)}\colon\R^M\to\R^M$ the orthogonal projection onto the range of $A$. An operator $B\colon\R^M\to\R^N$ is called a \emph{right-inverse} of $A$ if $ABy=y$  for all $y\in\range(A)$ and a \emph{least-squares right-inverse} of $A$ if, in addition, $ABy=P_{\range(A)}(y)$ for all $y\in\R^M$. 
\end{definition}

For a least-squares right-inverse, $By$ is a least-squares solution of the equation $Ax = y$ for every $y \in \R^M$, since $\|A(By) - y\|_2 = \|P_{\range(A)}(y) - y\|_2 = \min_x \|Ax - y\|_2$. The Moore-Penrose pseudoinverse $A^\dagger$ is the canonical example. We emphasize that a least-squares right-inverse prescribes the behavior of $AB$ on all of $\R^M$, but does \emph{not} require $B$ to map into $\ker(A)^\perp$.

\begin{condition}\label{cond:arch}
    Let $A\colon\R^N\to\R^M$ be a linear operator and $P_I$ be a projection operator on $\R^M$. Let further $X$ be a square-integrable RV on $\R^N$. Assume the following hold:
    \begin{enumerate}
        \item \label{cond:arch_1} $B_I\colon\R^M\to\R^N$ is a least-squares right-inverse of $P_I\circ A$.
        
        \item \label{cond:arch_2} $\noise_I$ denotes a RV on $\R^I$ independent of $AX$, with $\noise_\parallel \coloneqq P_{\range(P_I A)}(\noise_I)$ and
$\noise_\perp \coloneqq P_{\range(P_I A)^\perp}(\noise_I)$ independent of each other.
    \end{enumerate}
\end{condition}

Note that $A$ is not assumed to be injective; injectivity is needed only for the image-domain
statement in Corollary~\ref{cor:image}. Condition~\ref{cond:arch}.\ref{cond:arch_2} is satisfied for i.i.d.\ Gaussian noise, for which orthogonal projections onto complementary subspaces are independent. For non-Gaussian noise, uncorrelatedness of $\noise_\parallel$ and $\noise_\perp$ does not imply their independence in general. In particular, the Poisson noise used in our experiments satisfies the condition only
approximately: even in a high-dose regime, where it is well approximated by an additive Gaussian,
the variance depends on the signal, so that $\noise_I$ is not independent of $AX$;
see Section~\ref{sec:experiments}. 

\begin{theorem}\label{thm:arch}
Assume Condition~\ref{cond:arch} holds, and denote by
$
\Fac(A, B_I) := \bigl\{ A \circ f \circ B_I \;:\; f\colon\R^N \to \R^N \text{ measurable} \bigr\}
$
the class of all factorized reconstruction maps. Then minimizing the full self-supervised risk over $\Fac(A, B_I)$ recovers the conditional expectation:
\[
\mathbb{E}[A X \mid Y_I] =
\argmin_{g \in \Fac(A, B_I)}
\Esq{g(Y_I) - AX}.
\]
\end{theorem}

\begin{proof}
Write $\argmin_{g\in \Fac(A, B_I)} \Esq{g(Y_I)-AX} = A\circ f^*\circ B_I$ with $f^* = \argmin_{f} \Esq{Af(B_I(Y_I))-AX}$, the right-hand sides being representatives of the respective $\argmin$ sets.
The objective depends on $f$ only through the composed map $A\circ f$. As $f$ ranges over all
measurable maps $\R^N\to\R^N$, the quantity $A f(B_I(Y_I))$ ranges over all
$\sigma(B_I(Y_I))$-measurable, square-integrable RVs with values in $\range(A)$.
Since $\range(A)$ is a linear subspace and $AX$ takes values in it, so does
$\mathbb{E}[AX\mid B_I(Y_I)]$. Projection in $L^2$ onto the $\sigma(B_I(Y_I))$-measurable random
vectors therefore gives
\[
\argmin_{u \;\sigma(B_I(Y_I))\text{-meas.},\; u\in\range(A)}
\mathbb E\bigl[\|u - AX\|_2^2\bigr]
= \mathbb{E}[AX\mid B_I(Y_I)] ,
\]
and this value is attained by $f^\ast = A^{\dagger}\,\mathbb{E}[AX\mid B_I(Y_I) = \,\cdot\,]$, with
$A^\dagger$ the Moore--Penrose pseudoinverse, since $A A^\dagger = P_{\range(A)}$ acts as the
identity on $\range(A)$. Hence any minimizer satisfies $
A\circ f^*\circ B_I(Y_I) = \mathbb E[AX\mid B_I(Y_I)] \quad\text{almost surely}$. No injectivity of $A$ is used; without it $f^\ast$ is determined only up to $\ker(A)$, which is Corollary~\ref{cor:image}. It remains to show $\mathbb{E}[AX\mid B_I(Y_I)] = \mathbb{E}[AX\mid Y_I]$. 
Since $AX$ is independent of $\noise_I$ and $\noise_\parallel$ is independent of $\noise_\perp$, the triple $(AX,\noise_\parallel,\noise_\perp)$ has a product law. Hence $\noise_\perp$ is independent of $(AX,\noise_\parallel)$, therefore also of $(AX,P_I(AX)+\noise_\parallel)$, and
\[
\mathbb{E}[AX\mid Y_I]
=
\mathbb{E}[AX\mid P_I(AX)+\noise_\parallel,\noise_\perp]
=
\mathbb{E}[AX\mid P_I(AX)+\noise_\parallel]
=
\mathbb{E}[AX\mid P_{\range(P_IA)}(Y_I)].
\]
The least-squares right-inverse property \ref{cond:arch}.\ref{cond:arch_1} implies $\sigma(P_{\range(P_IA)}(Y_I))
\subseteq
\sigma(B_I(Y_I))
\subseteq
\sigma(Y_I)$, since $P_{\range(P_IA)}(Y_I) = (P_I A)\, B_I(Y_I)$ is a measurable function of $B_I(Y_I)$. This is where the least-squares property is needed: $Y_I$ lies outside $\range(P_I A)$ because of $\noise_\perp$, and the standard right-inverse property would not suffice (Remark~\ref{rem:counterexample}). Thus $\mathbb{E}[AX \mid Y_I]$ is $\sigma(B_I(Y_I))$-measurable, and the tower property yields
$
\mathbb{E}[AX \mid B_I(Y_I)]
=
\mathbb{E}\bigl[\, \mathbb{E}[AX \mid Y_I] \;\big|\; B_I(Y_I) \bigr]
=
\mathbb{E}[AX \mid Y_I]$,
which completes the proof.
\end{proof}

\begin{remark}[Minimal condition on $B_I$]\label{rem:minimal}
The least-squares right-inverse property \ref{cond:arch}.\ref{cond:arch_1} enters the proof only
through the inclusion $\sigma(P_{\range(P_I A)}(Y_I)) \subseteq \sigma(B_I(Y_I))$, which together
with \ref{cond:arch}.\ref{cond:arch_2} gives
$\mathbb{E}[AX \mid B_I(Y_I)] = \mathbb{E}[AX \mid Y_I]$. The minimal condition for that identity
is that $B_I(Y_I)$ be a sufficient statistic for $AX$, i.e.\ that $AX$ and $Y_I$ be conditionally
independent given $B_I(Y_I)$: $B_I$ must not discard information about $AX$ contained in $Y_I$.
\end{remark}

The next proposition quantifies what an inexact initial reconstruction costs, complementing Theorem~\ref{thm:stability}, which measures an inexact coverage condition.

\begin{proposition}[The price of the factorized class]\label{prop:price}
Let $A$ be linear, let $B_I \colon \mathbb{R}^M \to \mathbb{R}^N$ be measurable and let
$Z = \mathbb{E}[AX \mid Y_I]$. Then
\[
\min_{g \in \Fac(A, B_I)} \Esq{g(Y_I) - AX}
\;-\;
\min_{g \;\sigma(Y_I)\text{-meas.}} \Esq{g(Y_I) - AX}
\;=\;
\Esq{Z - \mathbb{E}\bigl[Z \bigm| B_I(Y_I)\bigr]} .
\]
The restriction to $\Fac(A, B_I)$ thus costs exactly the part of $Z$ that $B_I(Y_I)$ does not
determine, and nothing when $B_I(Y_I)$ is sufficient for $AX$.
\end{proposition}

\begin{proof}
By the argument in the proof of Theorem~\ref{thm:arch}, which uses no property of $B_I$, every
minimizer over $\Fac(A, B_I)$ satisfies $A f^\ast(B_I(Y_I)) = \mathbb{E}[AX\mid B_I(Y_I)]$
almost surely, and since $\sigma(B_I(Y_I)) \subseteq \sigma(Y_I)$ the tower property gives
$\mathbb{E}[AX\mid B_I(Y_I)] = \mathbb{E}[Z \mid B_I(Y_I)]$. The Pythagoras identity for conditional
expectations applied twice yields $
\Esq{g(Y_I) - AX} = \Esq{g(Y_I) - Z} + \sigma^2$ with  $\sigma^2 = \Esq{AX - Z} $ for every $\sigma(Y_I)$-measurable $g$. Taking $g = \mathbb{E}[Z\mid B_I(Y_I)]$ and $g = Z$ and
subtracting gives the claim.
\end{proof}

\begin{remark}[What the criterion targets]\label{rem:mmse}
All results of this section identify the conditional expectation, that is, the minimum mean squared
error estimator. This is the estimator that maximizes PSNR, and it is the natural target for a
criterion built on a squared loss, but it is a choice rather than a necessity: the conditional
expectation averages over all images consistent with $Y_I$ and is therefore systematically
smoothing. Where low-contrast detectability rather than mean squared error is the quantity of
interest, a different criterion, and correspondingly a different notion of optimality, would be
appropriate.  The equivalence between the partially self-supervised and the full self-supervised risk established here is independent of that choice in the following sense: it compares two ways of \emph{estimating} the same objective, and any weighted quadratic objective would give the same argument. The equivalence established here is independent of that choice: it compares two ways of \emph{estimating} the same objective, and any weighted quadratic objective gives the same argument.
\end{remark}

\begin{remark}[Necessity of the least-squares property]\label{rem:counterexample}
If \(B\) does not control $(P_I A) B$ outside $\range(P_I A)$, then \(B(Y_I)\) can contain less information about $P_I(AX)$ than \(Y_I\), and Theorem~\ref{thm:arch} fails. Let
\[
A=
\begin{pmatrix}
1\\
0
\end{pmatrix},
\qquad
P_I=I_2,
\qquad
B(x_1,x_2)=x_1+x_2 .
\]
Then \(P_IAX=AX=(X,0)^\top\), and \(B\) \emph{is} a right-inverse of $P_I A$ in the standard sense, since $(P_I A) B y = y$ for all $y = (y_1, 0)^\top \in \range(P_IA)$. However, \(B\) is not a least-squares right-inverse: for $y = (0, y_2)^\top \not\in \range(P_IA)$ we have $(P_I A) B y = (y_2, 0)^\top \neq 0 = P_{\range(P_IA)}(y)$.\\
Consider the observation model
$
Y_I = AX+\noise
=
(X+\noise_1, \noise_2)^\top $, where \(X,\noise_1,\noise_2\) are independent standard Gaussian random variables. Then \(Y_I\) contains the two independent noisy observations
\(X+\noise_1\) and \(\noise_2\), whereas $B(Y_I)=X+\noise_1+\noise_2 $ only contains their sum. Consequently, $\mathbb{E}[AX\mid Y_I] \neq \mathbb{E}[AX\mid B(Y_I)]$.
\end{remark}

For the Radon transform the FBP is not a right-inverse of $P_I\circ A$, not even in the standard sense: restricting $A$ by $P_I$, or merely discretizing it, removes information, so that $P_I A B(y)\neq y$ in general even for $y\in\range(P_IA)$.

Combining Theorem~\ref{thm:arch} with the equivariant construction of Section~\ref{sec:recon} closes the gap between the generic equivariant class of Theorem~\ref{thm:main} and the architecture used in practice.

\begin{corollary}[Equivariant architectures suffice]\label{cor:equivariant}
Assume Conditions~\ref{cond:main}, \ref{cond:ST} (with $B = B_I$), and \ref{cond:arch} hold, and denote by
$$
\Equi(A, B_I; \TS) := \bigl\{ A \circ f \circ B_I \;:\; f\colon\R^N \to \R^N \text{ measurable and } \TS\text{-equivariant} \bigr\}$$ the equivariant subclass. Then minimizing the full self-supervised risk over $\Equi(A, B_I; \TS)$ recovers the conditional expectation:
\[
\mathbb{E}[AX \mid Y_I]
=
\argmin_{g \in \Equi(A, B_I; \TS)}
\Esq{g(Y_I) - AX}.
\]
\end{corollary}

\begin{proof}
Let $f^*$ be the minimizer from the proof of Theorem~\ref{thm:arch}, so that $g^* \coloneqq A \circ f^* \circ B_I$ satisfies $g^*(Y_I) = \mathbb{E}[AX \mid Y_I]$ almost surely, and define the symmetrization
$\tilde f \coloneqq \frac{1}{L} \sum_{\ell=1}^{L} S_\ell^{-1} \circ f^* \circ S_\ell$,
which is $\TS$-equivariant by Theorem~\ref{thm:Xnet} (applied with $\Phi = f^*$). The intertwining relations of Condition~\ref{cond:ST} give $A S_\ell^{-1} = T_\ell^{-1} A$ and $S_\ell B_I = B_I T_\ell$, hence
\[
A \circ \tilde f \circ B_I
= \frac{1}{L} \sum_{\ell=1}^{L} (A S_\ell^{-1}) \circ f^* \circ (S_\ell B_I)
= \frac{1}{L} \sum_{\ell=1}^{L} T_\ell^{-1} \circ (A \circ f^* \circ B_I) \circ T_\ell
= \frac{1}{L} \sum_{\ell=1}^{L} T_\ell^{-1} \circ g^* \circ T_\ell .
\]
By \ref{cond:jointinv} and Lemma~\ref{lem:joint-to-ce} (with $Y = AX$), $g^*$ is $T_\ell$-equivariant $\pi_{Y_I}$-a.e., so $T_\ell^{-1}\circ g^*\circ T_\ell = g^*$ $\pi_{Y_I}$-a.e.\ for every $\ell$ and hence $A \circ \tilde f \circ B_I (Y_I) = \mathbb{E}[AX\mid Y_I]$ almost surely. The minimal value of Theorem~\ref{thm:arch} is therefore attained within $\Equi(A, B_I; \TS) \subseteq \Fac(A, B_I)$. A pointwise equivariant representative is provided by Lemma~\ref{lem:global-version}.
\end{proof}

\begin{corollary}[Main result]\label{cor:elipps}
Assume Conditions~\ref{cond:main}, \ref{cond:ST} (with $B = B_I$) and \ref{cond:arch} hold. Then
\[
\mathbb{E}[AX \mid Y_I]
=
\argmin_{g \in \Equi(A, B_I; \TS)}
\Esq{P_S\, g(Y_I) - Y_S} .
\]
\end{corollary}

\begin{proof}
By Theorem~\ref{thm:Xnet}, every $g = A\circ f\circ B_I$ with $\TS$-equivariant $f$ is
$\T$-equivariant, hence $\Equi(A, B_I; \TS) \subseteq \Equi(\mathbb{R}^M;\T)$. By
Corollary~\ref{cor:subclass} applied to $\mathcal{C} = \Equi(A, B_I; \TS)$, minimizing the
partially self-supervised noisy risk over this class is equivalent to minimizing
$\Esq{g(Y_I) - \mathbb{E}[AX\mid Y_I]}$ over the same class. By
Corollary~\ref{cor:equivariant} the value $\mathbb{E}[AX \mid Y_I]$ is attained within
$\Equi(A, B_I; \TS)$, so this minimum equals zero and is attained precisely at
$\mathbb{E}[AX \mid Y_I]$.
\end{proof}

Corollary~\ref{cor:elipps} closes the chain from the training objective to the optimal predictor:
no ground-truth image, no fully sampled measurement and no noise-free target enters, and yet the
minimizer over the trained architecture class is the conditional expectation given the
undersampled data. The following corollary transfers the statement from the measurement domain,
in which the loss is formulated, to the image domain, in which reconstructions are evaluated.

\begin{corollary}[Image-domain optimality up to the null space]\label{cor:image}
Under the assumptions of Corollary~\ref{cor:elipps}, let $g^\ast = A\circ f^\ast \circ B_I$ be a
minimizer of the partially self-supervised risk over $\Equi(A, B_I; \TS)$. Then the learned
reconstruction map satisfies
\[
P_{\ker(A)^\perp} f^\ast\bigl(B_I(Y_I)\bigr)
= P_{\ker(A)^\perp}\,\mathbb{E}[X \mid Y_I]
\qquad \text{almost surely}.
\]
If in addition $A$ is injective, then $f^\ast(B_I(Y_I)) = \mathbb{E}[X\mid Y_I]$ almost surely, i.e.\
ELIPPS training returns the minimum mean squared error estimator of the image given the
undersampled measurements.
\end{corollary}

\begin{proof}
By Corollary~\ref{cor:elipps} and linearity of $A$,
$A f^\ast(B_I(Y_I)) = \mathbb{E}[AX\mid Y_I] = A\,\mathbb{E}[X\mid Y_I]$ almost surely, so the
difference $f^\ast(B_I(Y_I)) - \mathbb{E}[X\mid Y_I]$ lies in $\ker(A)$; applying
$P_{\ker(A)^\perp}$ gives the first claim, and injectivity means $\ker(A) = \{0\}$.
\end{proof}

\begin{remark}[The null space is not identified by the risk alone]\label{rem:nullspace}
Corollary~\ref{cor:image} is sharp: the training risk depends on $A f \circ B_I$ only, so no criterion of this form determines the component in $\ker(A)$. For the discretized Radon transform of Section~\ref{sec:experiments}, which maps $361^2 = 130\,321$ pixels to $128 \times 361 = 46\,208$ measurements, the image-domain statement therefore holds only modulo $\ker(A)$. What constrains that component in practice is the hypothesis class: the equivariance of $f_\theta$, the choice of $B_I$ and the implicit bias of the network. Combining the sensing theorems of \cite{tachella2023sensing,chen2021equivariant} with the present setting is a natural next step.
\end{remark}

\begin{remark}[Biased supervision within the factorized class]\label{rem:bias-class}
For the constrained class $\Equi(A, B_I; \TS)$, Remark~\ref{rem:bias} needs one qualification: its elements take values in $\range(A)$, whereas $\mathbb{E}[AX+\beta(X)\mid Y_I]$ need not, so the attainability step of Corollary~\ref{cor:elipps} may fail. Consistency is unaffected, since by Corollary~\ref{cor:subclass} training returns the $L^2$-projection of the shifted target onto the realized class. The qualification is vacuous when $\range(A) = \R^M$, the expected situation in Section~\ref{sec:experiments}.
\end{remark}

\begin{remark}[Role of Condition~\ref{cond:arch}]\label{rem:role}
The two conditions play different roles. Condition~\ref{cond:main} concerns the \emph{training criterion}: by Corollary~\ref{cor:subclass} it makes the partially self-supervised risk rank equivariant candidates exactly as the full self-supervised risk does, for every hypothesis class. Condition~\ref{cond:arch} concerns the \emph{hypothesis class}: it keeps the optimum from being excluded by the factorized architecture. If it is violated, training still returns the best approximation of $\mathbb{E}[AX \mid Y_I]$ within the realized class, so misspecification causes an approximation error, never an inconsistent target. The experiments in Section~\ref{sec:experiments} operate in that regime.
\end{remark}

\begin{remark}[The two requirements on $B_I$ are compatible]\label{rem:compatible}
Condition~\ref{cond:ST}.2 and Condition~\ref{cond:arch}.\ref{cond:arch_1} constrain the same
operator, and it is not immediate that they can be met at once. They can, and the Moore--Penrose
pseudoinverse of the subsampled forward operator is the canonical example. Assume the $T_\ell$ and
$S_\ell$ are orthogonal, $T_\ell A = A S_\ell$, and $P_I T_\ell = T_\ell P_I$, and put
$M := P_I A$. Then
\[
M S_\ell = P_I A S_\ell = P_I T_\ell A = T_\ell P_I A = T_\ell M ,
\]
so $M = T_\ell M S_\ell^{-1}$. Since $(U M V)^\dagger = V^{\ast} M^\dagger U^{\ast}$ for unitary
$U, V$, this gives $M^\dagger = S_\ell M^\dagger T_\ell^{-1}$, that is
$M^\dagger T_\ell = S_\ell M^\dagger$. Hence $B_I := (P_I A)^\dagger$ intertwines as required by
Condition~\ref{cond:ST}.2 The same argument covers truncated Krylov methods on the normal equations, since $M^{\!\top} M S_\ell = S_\ell M^{\!\top} M$. In Section~\ref{sec:experiments} the $T_\ell$ are signed coordinate permutations and the $S_\ell$ coordinate permutations, so the assumptions of Corollary~\ref{cor:elipps} are jointly satisfiable; the FBP used there gives up \ref{cond:arch_1} for computational convenience, which by Remark~\ref{rem:role} costs accuracy and not consistency.
\end{remark}

\subsection{Summary of the method}
\label{sec:summary}

Suppose the image model satisfies Conditions~\ref{cond:main} and~\ref{cond:arch} with transforms satisfying Condition~\ref{cond:ST}, and let $(\Phi_\theta)_{\theta \in \Theta}$ be any image-space network. According to Theorem~\ref{thm:Xnet} the functions     $f_\theta = \tfrac{1}{L}\sum_{\ell} S_\ell^{-1} \circ \Phi_\theta \circ S_\ell$ and $g_\theta = A \circ f_\theta \circ B$ are $\TS$- and $\T$-equivariant, respectively. By Theorem~\ref{thm:arch} and Corollary~\ref{cor:equivariant} this restriction costs nothing when $B$ is a least-squares right-inverse of $P_I \circ A$, and Corollary~\ref{cor:elipps} gives the statement used in practice: the minimizer of the partially self-supervised noisy risk over this class is $\mathbb{E}[AX\mid Y_I]$, and by Corollary~\ref{cor:image} the reconstruction $f_\theta\circ B$ agrees with the minimum mean squared error estimator up to $\ker(A)$. For an approximate inverse such as the FBP, Remark~\ref{rem:role} shows that only the expressivity of the class is affected, not the correctness of the criterion. The training itself uses only the partial self-supervision data \eqref{eq:partial}, through 
\begin{equation}\label{eq:loss_elipps2}
\mathcal{L}_{\T}(\theta) = \Esq{P_S \circ g_\theta (Y_I) - Y_S} = \Esq{A_S \circ f_\theta \circ B (Y_I) - Y_S}.
\end{equation}
At test time, the reconstruction is simply obtained by the learned reconstruction maps $f_\theta \circ B$.

\section{Experiments}
\label{sec:experiments}

This section evaluates ELIPPS under severe partial self-supervision and low-dose noise. We consider two instantiations of the undersampled inverse problem \eqref{eq:ip} in CT, which differ only in the axis along which the sinogram is subsampled:
\begin{itemize}
    \item \textsc{Detector Upscaling:} only a subset of detector bins is observed; the goal is to recover the full-resolution measurement and a corresponding reconstruction.
    \item \textsc{Sparse-View Reconstruction:} only a subset of projection angles is observed; the goal is to recover a high-quality reconstruction from angularly undersampled measurements.
\end{itemize}
In both settings, we instantiate the transformation families $\T$ and $\TS$ of Section~\ref{sec:selfsup} via the eight symmetries of the dihedral group
\begin{equation}
D_4 = \{\id,\, r,\, r^2,\, r^3,\, \sigma,\, r\sigma,\, r^2\sigma,\, r^3\sigma\},
\end{equation}
where $r$ is the rotation by $90^\circ$ and $\sigma$ the reflection about the vertical image axis.
The group $D_4$ acts on the image domain by rotations and reflections of the pixel grid and, via the intertwining relation with the Radon transform, on the measurement domain by permutations of the sinogram coordinates. This makes the abstract framework of Section~\ref{sec:selfsup} fully explicit: we verify Conditions~\ref{cond:main} and~\ref{cond:ST} for this instantiation below.

\subsection{Dataset and Acquisition Model}

We use a $1\%$ subset of the LoDoPaB dataset \cite{leuschner2021lodopab}, consisting of $358$ training and $35$ test images. Unless stated otherwise (see the ablation in Section~\ref{sec:ablation}), the training set is augmented by the full $D_4$ group, yielding an $8\times$ larger training set whose empirical distribution is $D_4$-invariant by construction.
All images are resized to $361 \times 361$ pixels and multiplied by a circular disk mask restricting the reconstruction to the field of view. Here $361$ is the side length of the
grid, so the image has $361^2 = 130\,321$ pixels. Both the side length and the number of detector
bins are odd, so that a single central pixel and a single central detector bin exist and are fixed
by the discrete $D_4$ action; this is what makes the rotations and reflections exact permutations
of the grids. The circular field of view is the disc inscribed in that grid and is therefore $361$
pixels across. The disk mask is $D_4$-invariant and therefore compatible with the invariance assumptions of Condition~\ref{cond:main}.

We adopt a parallel-beam geometry with $N_\theta = 128$ projection angles on the endpoint-exclusive
grid $\theta_j = j\pi/N_\theta$, $j = 0,\dots,N_\theta-1$, and $N_s = 361$ detector bins, yielding
dense sinograms of size $128\times 361$.  The angular spacing $\pi/N_\theta$ is chosen such that $\pi/2$ is an integer multiple of it, which is what makes the rotations of $D_4$ exact permutations of the sinogram grid; this fails, for instance, for $N_\theta = 129$. To simulate low-dose acquisition we use the count model of the LoDoPaB benchmark.  With base photon count $\lambda = 1000$ and attenuation factor $\mu = 2.76$, the dimensionless product of the linear attenuation coefficient of the LoDoPaB forward model with the field of view under our normalization of $Ax$, the detected counts along each ray are
\begin{equation}\label{eq:countmodel}
N \sim \mathcal{P}\bigl(\lambda\, e^{-\mu A x}\bigr),
\qquad
y^\delta = -\tfrac{1}{\mu}\,\log\bigl(\max\{N,1\}/\lambda\bigr) ,
\end{equation}
where the clipping of the counts from below keeps the logarithm well defined. All Poisson-corrupted
sinograms below are generated by \eqref{eq:countmodel}.

\subsection{$D_4$-Equivariance and Explicit Operator Relations}

We now instantiate $\T = (T_\ell)_{\ell=1}^8$ in the measurement domain and the corresponding image-domain transforms $\TS = (S_\ell)_{\ell=1}^8$, and verify the intertwining relations of Condition~\ref{cond:ST}.

\paragraph{Radon transform and orthogonal transforms:}
Let the parallel-beam Radon transform be discretized as
$A \in \mathbb{R}^{(N_\theta N_s)\times(N_x N_y)}$, so that $Ax \in \mathbb{R}^{N_\theta \times N_s}$
is the sinogram of $x \in \mathbb{R}^{N_x \times N_y}$. Let $Q \in O(2)$ map the image grid onto
itself, so that $i \mapsto Q^{-1} i$ is a permutation of the pixel indices, and write $S_Q$ for the
corresponding permutation matrix, $(S_Q x)_i = x_{Q^{-1} i}$. Assume further that $Q$ maps the
sinogram grid onto itself, through $n_{\theta'} = Q^\top n_\theta$ with
$n_\theta = (\cos\theta,\sin\theta)^\top$ and, whenever $\theta' \notin [0,\pi)$, the Radon symmetry
$(Ax)_{\theta+\pi,\,s} = (Ax)_{\theta,\,-s}$; write $T_Q$ for the resulting permutation matrix. A
change of variables in the continuous Radon transform then gives
\begin{equation}\label{eq:radon_equivariance_exp}
A S_Q = T_Q A .
\end{equation}
Both assumptions hold for the elements of $D_4$ with the grid sizes fixed above, since $N_s = 361$ is odd and $\pi/2$ is an integer multiple of $\pi/N_\theta$ with $N_\theta = 128$.

\paragraph{$D_4$ instantiation:}
Let $D_4 \subseteq O(2)$ consist of the four rotations $R_{k\pi/2}$ and the four reflections
$R_{k\pi/2}F$, $k \in \{0,1,2,3\}$, where
\[
R_\alpha =
\begin{pmatrix}
\cos\alpha & -\sin\alpha\\
\sin\alpha & \cos\alpha
\end{pmatrix},
\qquad
F =
\begin{pmatrix}
-1 & 0\\
0 & 1
\end{pmatrix}.
\]
Each element maps both grids onto itself, so that \eqref{eq:radon_equivariance_exp} applies.
Indexing them as $Q_1,\dots,Q_8$ and writing $S_\ell := S_{Q_\ell}$ and $T_\ell := T_{Q_\ell}$ gives
the families $\TS$ and $\T$ of Condition~\ref{cond:ST}. A rotation by $k\pi/2$ shifts the angular
axis, a reflection reverses it:
 $ (T_k\, y)_{\theta,s} = y_{\theta - k\pi/2,\; s}$ and $(T_{k+4}\, y)_{\theta,s} = y_{\pi - \theta + k\pi/2,\; s}$. Any index that leaves $[0,\pi)$ is brought back by $y_{\theta+\pi,\,s} = y_{\theta,\,-s}$, which flips the detector axis. All indices are modulo the grid size.

\paragraph{Verification of Condition~\ref{cond:ST}:}
The relations above yield the intertwining $T_\ell A = A S_\ell$ for $\ell = 1,\dots,8$, i.e., Condition~\ref{cond:ST}.1. Since $(S_\ell)_{\ell=1}^8$ is the permutation representation of $D_4$ and hence a group of order eight, Condition~\ref{cond:ST}.3 holds as well. Finally, the filtered backprojection intertwines with rotations and reflections of the image grid, since both the backprojection and the (radially symmetric) filtering step commute with the induced sinogram permutations; this gives $B T_\ell = S_\ell B$ and thus Condition~\ref{cond:ST}.2.

\subsection{Observation and Masking}

\paragraph{Acquisition model:}
For each training sample, we generate two conditionally independent noisy full sinograms according
to \eqref{eq:countmodel}, written $\mathcal{P}_1(Ax)$ and $\mathcal{P}_2(Ax)$ for brevity, and apply
the input and supervision sampling operators separately:
\begin{equation}
y_I^\delta
=
P_I\left(\mathcal{P}_1(Ax)\right),
\qquad
y_S^\delta
=
P_S\left(\mathcal{P}_2(Ax)\right),
\qquad
\mathcal{P}_1 \perp \mathcal{P}_2 \mid x.
\label{eq:acquisition}
\end{equation}
Here $I$ denotes the task-dependent input sampling set, while $S$ is the fixed supervision set
defined below. Because the input and supervision measurements are generated from independent Poisson
realizations, their conditional independence is preserved even when $I\cap S\neq\varnothing$. As
shown in Remark~\ref{rem:A4}, this conditional independence is the structure required by
\ref{cond:noise}, and it is compatible with signal-dependent noise; the remaining discrepancy in our
simulation is the bias of the log transform, quantified below. Only $P_I$ of the first realization and $P_S$ of the second are ever used, so the training acquisition is cheaper than two full scans. The relevant count is projections rather than measurements, since a projection irradiates the whole slice: for detector upscaling the input uses all $128$ angles and the supervision a further $32$, i.e.\ $1.25$ times the deployment scan, whereas for sparse-view reconstruction the input and supervision angles overlap in only $8$ of $32$, so that $64$ projections are needed and the factor is exactly two. A disjoint design $I \cap S = \emptyset$ would give the conditional centring of
\ref{cond:noise} from a single scan, because Poisson noise is independent across detector
coordinates, but it is incompatible with the remaining assumptions: if $P_I T_\ell = T_\ell P_I$ for
every $\ell$, then $I$ is a union of $\T$-orbits, while \ref{cond:frame} applied to $v = e_p$ for
$p \in I$ gives $\sum_\ell \|P_S T_\ell e_p\|_2^2 = c > 0$ and hence $T_\ell\, p \in S$ for some
$\ell$; since $T_\ell\, p$ lies in the orbit of $p$, it lies in $I$, so that
$I \cap S \neq \emptyset$. The two structural assumptions of Condition~\ref{cond:main} therefore
force the input and supervision sets to overlap.

What does realize partial self-supervision from a single scan, at unchanged total photon budget and
also on an overlap, is Poisson thinning: splitting the detected count $N$ as $N = N_1 + N_2$ with
$N_1 \mid N \sim \mathrm{Bin}(N,1/2)$ yields two exactly independent Poisson variables of half the
rate, from which $y_I^\delta$ and $y_S^\delta$ are formed by \eqref{eq:countmodel} with $\lambda$
replaced by $\lambda/2$. The reported experiments use \eqref{eq:acquisition}, which has the same structure at the increased dose quantified above. Thinning halves the counts in both measurements and therefore changes
the operating point of every method compared, so all baselines would have to be retrained; we
describe it as the dose-neutral realization of the setting rather than reporting it here.

\paragraph{Poisson noise and the log transform:}
Under the count model \eqref{eq:countmodel} the supervision measurement is not unbiased: a
second-order expansion of the logarithm around the mean count gives
\[
\mathbb{E}\bigl[y^\delta \bigm| x\bigr] = Ax + \beta(x),
\qquad
\beta(x) \approx \frac{1}{2\mu\,\lambda_{\mathrm{eff}}(x)},
\qquad
\lambda_{\mathrm{eff}}(x) = \lambda\, e^{-\mu Ax},
\]
so that \ref{cond:noise} holds only up to $\beta$. The clipping of zero counts is inactive at this dose: the smallest effective count is $\lambda e^{-\mu} \approx 63$, so that $\mathbb{P}(N = 0)$ is of order $10^{-28}$ and no zero count occurs anywhere in the data; it is retained for numerical safety and starts to contribute below $\lambda \approx 200$. The bias is largest where the expected count is smallest. On the data set used here the
sinograms are normalized by their maximum, so that $Ax \in [0,1]$ with mean $0.36$ on the
training and $0.44$ on the test split; the corresponding effective counts are
$\lambda_{\mathrm{eff}} \approx 370$ and $\approx 290$, giving
$\beta \approx 4.9\cdot 10^{-4}$ and $\beta \approx 6.2\cdot 10^{-4}$. For the most strongly
attenuated rays, with $\lambda_{\mathrm{eff}} = \lambda e^{-\mu} \approx 63$, it reaches
$\beta \approx 2.9\cdot 10^{-3}$; the normalization makes this worst case attained. Even the worst case is more than two orders of magnitude below
the attenuation values it perturbs. By Remark~\ref{rem:bias} this does not invalidate the training criterion: Theorem~\ref{thm:main}
then identifies $\mathbb{E}[AX + \beta(X)\mid Y_I]$ rather than $\mathbb{E}[AX\mid Y_I]$, so the bias
of the supervision measurement is inherited by the learned predictor. For the architecture actually
trained this holds verbatim as long as $\mathbb{E}[\beta(X)\mid Y_I]$ lies in $\range(A)$, which is
expected here because $A$ maps $130\,321$ pixels to $46\,208$ measurements and is therefore expected
to have full row rank; otherwise the target is the $L^2$-projection of
$\mathbb{E}[AX+\beta(X)\mid Y_I]$ onto the realized class, and the deviation from
$\mathbb{E}[AX\mid Y_I]$ stays bounded by $\|\mathbb{E}[\beta(X)\mid Y_I]\|$.
The exactness is thus a property of \ref{cond:noise}, which the pre-log Poisson model satisfies exactly and the count model \eqref{eq:countmodel} up to the bias just quantified; a bias-corrected target removes it entirely.

Both tasks use this same noisy acquisition model and differ only in the axis along which the input sinogram $y_I^\delta$ is subsampled:
\begin{itemize}
\item \textsc{Detector Upscaling:} the detector coordinate $s$ is subsampled with stride $4$;
\item \textsc{Sparse-View Reconstruction:} the angular coordinate $\theta$ is subsampled with stride $4$.
\end{itemize}

\paragraph{From subsampled data to full-size inputs:}
Throughout Section~\ref{sec:selfsup}, $y_I$ is identified with its zero-embedding in $\mathbb{R}^M$. In the implementation we replace zero-filling by a linear extension operator $\Eop \colon \mathbb{R}^I \to \mathbb{R}^M$, which interpolates linearly along the subsampled axis ($s$ for detector upscaling, $\theta$ for sparse-view reconstruction). The initial reconstruction is $x^0 = \Bop\bigl(\Eop(y_I^\delta)\bigr)$ with $\Bop$ the filtered backprojection. Hence $\Bop \circ \Eop$ plays the role of the initial
reconstruction map $B_I$ in Condition~\ref{cond:arch}. Since $B_I = \Bop\circ\Eop$, the intertwining
requirement of Condition~\ref{cond:ST}.2 must be verified for $\Eop$ as well as for $\Bop$. It
holds: the retained index sets are the stride-$4$ lattices
$\{0,4,\dots,124\}$ in $\theta$ and $\{0,4,\dots,360\}$ in $s$, and both are invariant under the
index maps induced by $D_4$, namely the angular shift by $64$ grid points (a rotation by $\pi/2$ on the
grid $\theta_j = j\pi/128$), the angular reversal
$j \mapsto 128-j$ and the detector reversal $i \mapsto 360-i$, because $128$ and $360$ are divisible
by $4$. Linear interpolation along a subsampled axis uses a symmetric two-point stencil with weights
depending only on the distance to the neighbouring retained indices, and these distances are
preserved by shifts and reversals of the lattice; hence $\Eop$ commutes with the induced
permutations and $\Eop \circ T_\ell = T_\ell \circ \Eop$ on $\mathbb{R}^I$. Combined with
$\Bop T_\ell = S_\ell \Bop$ this gives $B_I T_\ell = S_\ell B_I$. As discussed in Remark~\ref{rem:counterexample}, FBP-based maps are not least-squares right-inverses of $P_I \circ A$ (not even right-inverses in the standard sense), so Condition~\ref{cond:arch}.\ref{cond:arch_1} is satisfied only approximately in our experiments; the results below indicate that the method is robust with respect to this violation.

\paragraph{Supervision mask $P_S$:}
The supervision set $S$ selects one quarter of the projection angles and one half of the detector bins,
\[
    S = \Theta_0 \times S_0,
    \qquad
    \Theta_0 = \{\theta : \theta \in [0, \pi/4)\},
    \qquad
    S_0 = \{s : s \geq 0\},
\]
covering an area fraction of $|S| / (N_\theta N_s) = (32/128) \times (181/361) \approx 1/8$ 
of the sinogram, and the transformed sets $\{T_\ell(S)\}_{\ell=1}^8$ cover the sinogram grid up to the exceptional set determined below.

\paragraph{Coverage of the supervision set:}
On the discrete grid, the covering by $\{T_\ell(S)\}_{\ell=1}^{8}$ is not exactly one-to-one,
because the $D_4$-action has fixed points. This is visible already at the level of cardinalities:
$8\,|S| = 8\cdot 32\cdot 181 = 46\,336$, whereas the grid has $N_\theta N_s = 128\cdot 361 = 46\,208$
points. Writing
\[
n(\theta,s) := \#\{\ell : (\theta,s) \in T_\ell(S)\} ,
\qquad
\sum_{\ell=1}^{8} \|P_S T_\ell v\|_2^2 = \sum_{\theta,s} n(\theta,s)\, v_{\theta,s}^2 ,
\]
an explicit evaluation of the eight permutations gives $n \equiv 1$ on $96.6\%$ of the grid,
together with the following exceptional set: the two projection angles $\theta \in \{0,\pi/2\}$ and
the central detector column $s = 0$ are covered twice, and their intersection four times, because
they are fixed by a reflection; the orbit $\theta \in \{\pi/4, 3\pi/4\}$ is not covered at all,
because it is disjoint from the half-open sector $\Theta_0 = [0,\pi/4)$, and on those two lines
$s = 0$ is of course not covered either. Explicitly, $n \equiv 1$ on $44\,640$ coordinates,
$n = 2$ on $844$, $n = 4$ on $2$ and $n = 0$ on $722$, which sums to $8\,|S| = 46\,336$. Condition \ref{cond:frame} therefore holds in the approximate form \eqref{eq:frameapprox} with $c_1 = 1$, $c_2 = 4$ and $Q$ the coordinate projection onto the complement of the two uncovered angular lines. Theorem~\ref{thm:stability} applies with $c_1 = 1$ and therefore identifies $\mathbb{E}[AX\mid Y_I]$ exactly on the range of $Q$, that is on $45\,486$ of the $46\,208$ sinogram coordinates, or $98.4\%$; the uneven multiplicities on the fixed lines reweight the risk but do not move its minimizer. On the two uncovered angular lines the masked risk is entirely unconstrained, so whatever accuracy is observed there is due to the hypothesis class rather than to the training criterion.

Figure~\ref{fig:tiling} shows the eight images $T_\ell(S)$ together with the exceptional set.  Up to that set they tile the sinogram domain, each point being covered exactly once; this is what \ref{cond:frame} asks for. Exact coverage is restored within the weighted formulation of Remark~\ref{rem:weights}: taking the
closed sector $\Theta_0 = [0,\pi/4]$ and the diagonal weight $W_{(\theta,s)} = a(\theta)\,b(s)$ with
$a(\theta) = 1/2$ for $\theta \in \{0,\pi/4\}$ and $a \equiv 1$ otherwise, and $b(0) = 1/2$ and
$b \equiv 1$ otherwise, one obtains
$\sum_{\ell=1}^{8}\|W^{1/2} P_S T_\ell v\|_2^2 = \|v\|_2^2$ for every $v$, i.e.\ \eqref{eq:frameW}
with $c = 1$. The experiments reported below use the unweighted mask $P_S$ with the half-open
sector; the two formulations differ on $3.4\%$ of the sinogram coordinates.

\begin{figure}[t]
\centering
\begin{tikzpicture}[
  font=\small,
  scale=1.0,
  tile/.style = {draw=elippsblue!45, fill=white},
  lab/.style  = {font=\scriptsize}
]
\def\Wd{12}\def\Ht{3}

\foreach \i in {0,1,2,3}{
  \draw[tile] (\i*3,0) rectangle (\i*3+3,\Ht);
  \draw[tile] (\i*3,-\Ht) rectangle (\i*3+3,0);
}
\fill[elippsmid!55] (0,0) rectangle (3,\Ht);
\draw[draw=elippsblue, very thick] (0,0) rectangle (3,\Ht);

\node at (1.5, 1.5)   {$\mathbf{S}$};
\node at (4.5, 1.5)   {$r^{3}\sigma$};
\node at (7.5, 1.5)   {$r^{3}$};
\node at (10.5,1.5)   {$\sigma$};
\node at (1.5,-1.5)   {$r^{2}$};
\node at (4.5,-1.5)   {$r\sigma$};
\node at (7.5,-1.5)   {$r$};
\node at (10.5,-1.5)  {$r^{2}\sigma$};

\draw[elippsaccent, very thick] (0,-\Ht) -- (0,\Ht);
\draw[elippsaccent, very thick] (6,-\Ht) -- (6,\Ht);
\draw[elippsaccent, very thick] (0,0) -- (3,0);
\draw[elippsaccent, very thick] (3,0) -- (9,0);
\draw[elippsaccent, very thick] (9,0) -- (\Wd,0);
\draw[densely dotted, very thick, black!70] (3,-\Ht) -- (3,\Ht);
\draw[densely dotted, very thick, black!70] (9,-\Ht) -- (9,\Ht);

\draw[-{Latex[length=2mm]}] (0,-\Ht-0.35) -- (\Wd+0.6,-\Ht-0.35) node[right] {$\theta$};
\foreach \xx/\lbl in {0/{$0$}, 3/{$\pi/4$}, 6/{$\pi/2$}, 9/{$3\pi/4$}, 12/{$\pi$}}{
  \draw (\xx,-\Ht-0.25) -- (\xx,-\Ht-0.45);
  \node[lab, below] at (\xx,-\Ht-0.45) {\lbl};
}
\draw[-{Latex[length=2mm]}] (-0.45,-\Ht) -- (-0.45,\Ht+0.6) node[above] {$s$};
\foreach \yy/\lbl in {-3/{$-s_{\max}$}, 0/{$0$}, 3/{$s_{\max}$}}{
  \draw (-0.35,\yy) -- (-0.55,\yy);
  \node[lab, left] at (-0.6,\yy) {\lbl};
}

\begin{scope}[shift={(0,-\Ht-1.55)}]
  \fill[elippsmid!55, draw=elippsblue] (0,0) rectangle (0.55,0.3);
  \node[lab, right] at (0.65,0.15) {supervision set $S=\Theta_0\times S_0$};
  \draw[elippsaccent, very thick] (6.3,0.15) -- (6.85,0.15);
  \node[lab, right] at (6.95,0.15) {covered twice (fixed by a reflection)};
  \draw[densely dotted, very thick, black!70] (0,-0.55) -- (0.55,-0.55);
  \node[lab, right] at (0.65,-0.55) {not covered: $\theta\in\{\pi/4,\,3\pi/4\}$};
\end{scope}
\end{tikzpicture}
\caption{The eight images $T_\ell(S)$ of the supervision set (blue) under the discrete
$D_4$-action on the sinogram domain $[0,\pi)\times[-s_{\max},s_{\max}]$, each labelled by the
group element that produces it ($r$ a rotation by $90^\circ$, $\sigma$ a reflection). Together
they cover the domain, up to the two exceptional sets marked in the legend: the lines drawn in
orange are covered twice, the dotted lines not at all. The exact multiplicities and their
consequences are discussed in Section~\ref{sec:experiments}.}
\label{fig:tiling}
\end{figure}
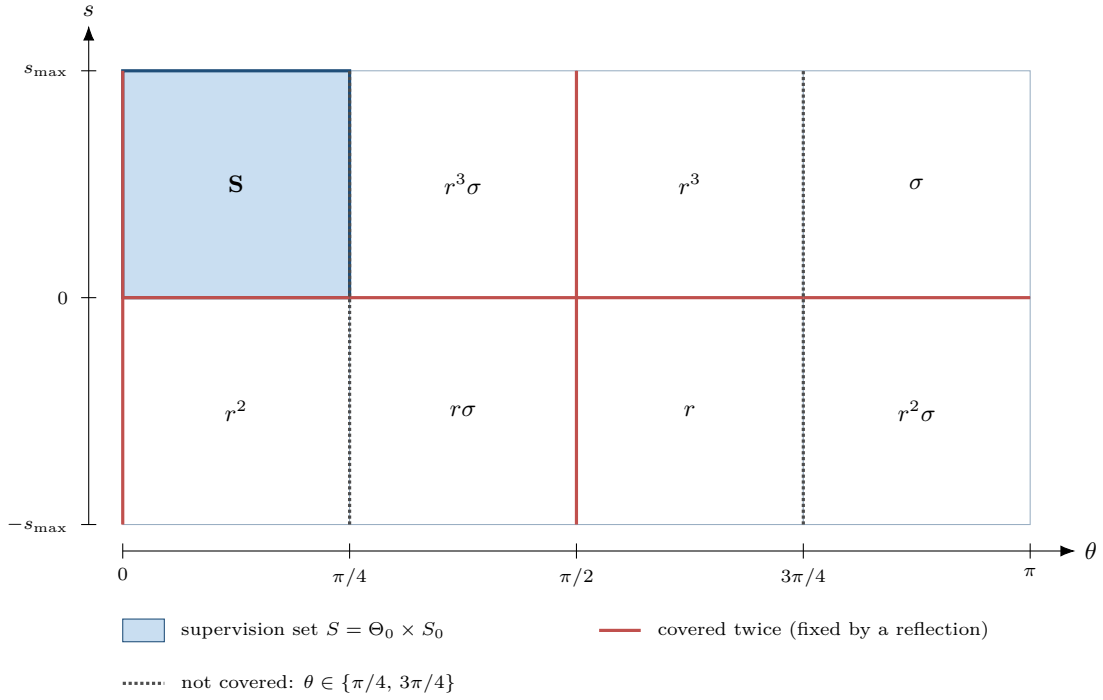

\paragraph{Verification of Condition~\ref{cond:main}:}
Augmenting the training distribution by $D_4$ makes it invariant by construction, so that
$S_\ell X \stackrel{d}{=} X$ for $\ell = 1, \dots, 8$ holds exactly for the distribution the network
is trained on, and Theorem~\ref{thm:main} identifies the conditional expectation for that
distribution. By the intertwining relation,
$
    T_\ell A X = A (S_\ell X) \stackrel{d}{=} A X $,
so $AX$ is $T_\ell$-invariant. Since each $T_\ell$ is a coordinate permutation and Poisson noise acts independently per coordinate, the Poisson corruption is conditionally equivariant,
$
    T_\ell\, \mathcal{P}(Ax) \ \big|\ Ax
    \ \stackrel{d}{=}\
    \mathcal{P}(T_\ell Ax) \ \big|\ T_\ell Ax$.
Moreover, the grid sizes $N_\theta = 128$ and $N_s = 361$ are chosen such that the stride-$4$ subsampling lattice is preserved by the $D_4$-action, hence $T_\ell P_I = P_I T_\ell$. Combining these properties yields
\[
    \bigl(T_\ell\, y_I^\delta,\; T_\ell\, A x\bigr)
    \ \stackrel{d}{=}\
    \bigl(y_I^\delta,\; A x\bigr),
\]
i.e.\ the joint invariance \ref{cond:jointinv}. The coverage condition \ref{cond:frame} holds with $c = 1$ in the weighted form \eqref{eq:frameW}, and in the approximate form \eqref{eq:frameapprox} for the unweighted mask actually used.

\paragraph{What augmentation does and does not change:}
Augmentation enforces \ref{cond:jointinv} rather than verifying it, and the two levels at which this matters should be separated. On the level of distributions,  writing $\pi_X$ for the law of the images in the data set and $\bar\pi_X := \frac{1}{8}\sum_{\ell=1}^{8} \pi_{S_\ell(X)}$ for its $D_4$-symmetrization,
the distribution the network is trained on is $\bar\pi_X$, so the
results of Section~\ref{sec:selfsup} identify the conditional expectation under $\bar \pi_X$. If the underlying
population is not $D_4$-invariant, and clinical CT images do have a canonical orientation, this need
not coincide with the conditional expectation under $\pi_X$; evaluation on the original test images is
therefore a generalization test beyond the invariance assumption rather than an instance of it.

On the level of the training objective, however, no such gap arises. Since the augmented training
set is closed under the group action, Proposition~\ref{prop:orbit} applies to each original image
separately: for every equivariant candidate, the masked loss summed over the eight augmented copies
of an image equals $c$ times the \emph{full, unmasked} loss on that image. The empirical objective
being minimized is thus exactly the full self-supervised empirical objective on the $358$ original
images, up to the noise cross term that vanishes in the mean. Augmentation is therefore not a distortion of the training target but the finite mechanism that realizes the coverage condition; what it changes is the population to which the estimator is optimal.

\paragraph{Relation to the noise model:}
The conditional independence of the two acquisitions is what brings the setup close to \ref{cond:noise}; no Gaussian approximation of the Poisson statistics is used anywhere in the analysis.

What is satisfied only approximately is Condition~\ref{cond:arch}: the FBP-based initial
reconstruction is not a least-squares right-inverse of $P_I \circ A$ (Remark~\ref{rem:counterexample}),
and the noise decomposition \ref{cond:arch}.\ref{cond:arch_2} holds exactly for i.i.d.\ Gaussian
rather than for Poisson noise. By Remark~\ref{rem:role} this does not affect the correctness of the
training criterion; it only means that $\mathbb{E}[AX \mid Y_I]$ need not lie in the realized
hypothesis class, in which case training returns its best approximation therein. Together with the
coverage defect quantified above, the experiments therefore serve as a robustness study for
precisely those quantities that the theory identifies as the sources of approximation error:
ELIPPS reaches, despite both violations, the same order of accuracy as the supervised reference model; the precise reading of that comparison is discussed in Section~\ref{sec:results-discussion}.

\subsection{Implementation Details}

All learned models use a U\mbox{-}Net backbone $\Phi_\theta$ with base width $32$ and three down-/up-sampling stages. We train with Adam (learning rate $10^{-4}$), gradient clipping $\|\nabla_\theta \loss\|_2 \leq 1$, and batch size $8$. Training runs for $1000$ epochs ($2000$ for the ablation of Section~\ref{sec:ablation}), and
the reported figures are those of the weights after the final epoch. No model selection is performed: the data set is split into training and test images
only, so selecting a checkpoint by test performance would tune on the test set, and we therefore do
not do it. Test evaluation is carried out every $10$ epochs for monitoring only.

\paragraph{Equivariant architecture:}
Following Theorem~\ref{thm:Xnet}, the reconstruction network is the $D_4$-symmetrized image-space network
\[
f_\theta = \frac{1}{8} \sum_{\ell=1}^{8} S_\ell^{-1} \circ \Phi_\theta \circ S_\ell ,
\qquad
\hat{x} = f_\theta(x^0) = f_\theta\bigl(\Bop(\Eop(y_I^\delta))\bigr),
\]
which is $\TS$-equivariant by construction. Training minimizes the empirical counterpart of the partially self-supervised risk \eqref{eq:loss_elipps2},
\[
\loss(\theta)
=
\sum_{\text{training pairs}}
\bigl\| P_S\, A f_\theta\bigl(\Bop(\Eop(y_I^\delta))\bigr) - y_S^\delta \bigr\|_2^2 .
\]

\subsection{Compared Methods}

No existing self-supervised method addresses the setting \eqref{eq:partial} with a fixed supervision set: full self-supervision requires measurements unavailable here, while sparse self-supervision and equivariant imaging use $y_I$ alone and cannot exploit $y_S$. A direct comparison would either grant them data they cannot use or deprive ELIPPS of the data that defines its setting. We therefore compare against the reference points proper to this setting:
\begin{enumerate}
    \item \textbf{FBP:} the model-based reconstruction $x^{\mathrm{FBP}} = x^0 = \Bop\bigl(\Eop(y_I^\delta)\bigr)$, followed by disk masking and clamping to $[0,1]$.
    \item \textbf{Naive learned refinement:} a plain (non-symmetrized) U\mbox{-}Net refinement of $x^{\mathrm{FBP}}$, trained with the masked sinogram loss
    $\loss_{\mathrm{naive}}(\theta) = \sum \| P_S\, A \Phi_\theta(x^0) - y_S^\delta \|_2^2$,
    i.e., the same supervision data as ELIPPS but without equivariant symmetrization.
    \item \textbf{Naive learned refinement with $D_4$ test-time augmentation:} the same
    non-symmetrized network as in 2., but evaluated as
    $\tfrac18\sum_{\ell} S_\ell^{-1}\Phi_\theta(S_\ell x^0)$, i.e.\ averaged over the eight
    group elements at test time only. Since the symmetrized network $f_\theta$ is simultaneously
    an eight-fold ensemble at inference, this baseline separates the effect of enforcing
    equivariance \emph{during training} from the generic averaging that symmetrization also
    provides at test time.
    \item \textbf{Fully supervised reference:} an oracle model trained with full clean sinograms ($P_S = \id$, noise-free targets), minimizing $\sum \| A \Phi_\theta(x^0) - A x \|_2^2$. We call this model \emph{supervised} for brevity; note that its targets are the clean full measurements $Ax$ rather than ground-truth images, so in the taxonomy of Section~\ref{sec:intro} it is the full self-supervised oracle, which is the strongest reference available in the measurement domain.
\end{enumerate}

It is a reference point, not an upper bound: optimization, architecture and sample size need not order the two models in either direction.

\subsection{Results}
\label{sec:results-discussion}

Table~\ref{tab:combined_results} reports the quantitative performance on both tasks. The proposed equivariant model outperforms the naive refinement baseline by a wide margin and achieves performance close to that of the fully supervised model. 

For detector upscaling, ELIPPS obtains $31.26$~dB PSNR and $0.823$ SSIM, compared with $31.47$~dB and $0.827$ for the supervised model. For sparse-view reconstruction, ELIPPS achieves $30.79$~dB PSNR and $0.816$ SSIM, compared with $31.07$~dB and $0.820$ for the supervised model. Thus, ELIPPS remains close to the fully supervised reference, with PSNR gaps of only $0.21$~dB and $0.28$~dB for detector upscaling and sparse-view reconstruction, respectively. The comparison with the
test-time-augmented baseline separates the two mechanisms at work: averaging the naive model over
the group at inference already gains $0.80$~dB and $0.62$~dB over the plain naive model, while
enforcing equivariance during training adds a further $1.49$~dB and $1.30$~dB on top of that. The
larger part of the margin is therefore due to the training criterion and not to the eight-fold
ensemble that symmetrization also provides at test time. Supervision over only $1/8$ of the sinogram, combined with the $D_4$ symmetry, thus suffices to
recover globally consistent measurements, in accordance with Theorem~\ref{thm:main}. These numbers
establish the order of accuracy reached under partial self-supervision; they come from a single training
run per configuration, so the comparison should be read at that resolution rather than as a ranking
to within hundredths of a decibel.

The evolution of the test-set metrics over training is shown in
Figs.~\ref{fig:d4_detector_metrics} and \ref{fig:d4_sparse_metrics}.
For both tasks, ELIPPS consistently outperforms the naive learned baseline
and approaches the fully supervised reference. The curves also show that
the reported final-epoch values are representative of the performance
reached near the end of training.

\begin{figure}[htb!]
    \centering

    \begin{subfigure}[t]{0.48\linewidth}
        \centering
        \includegraphics[width=\linewidth]{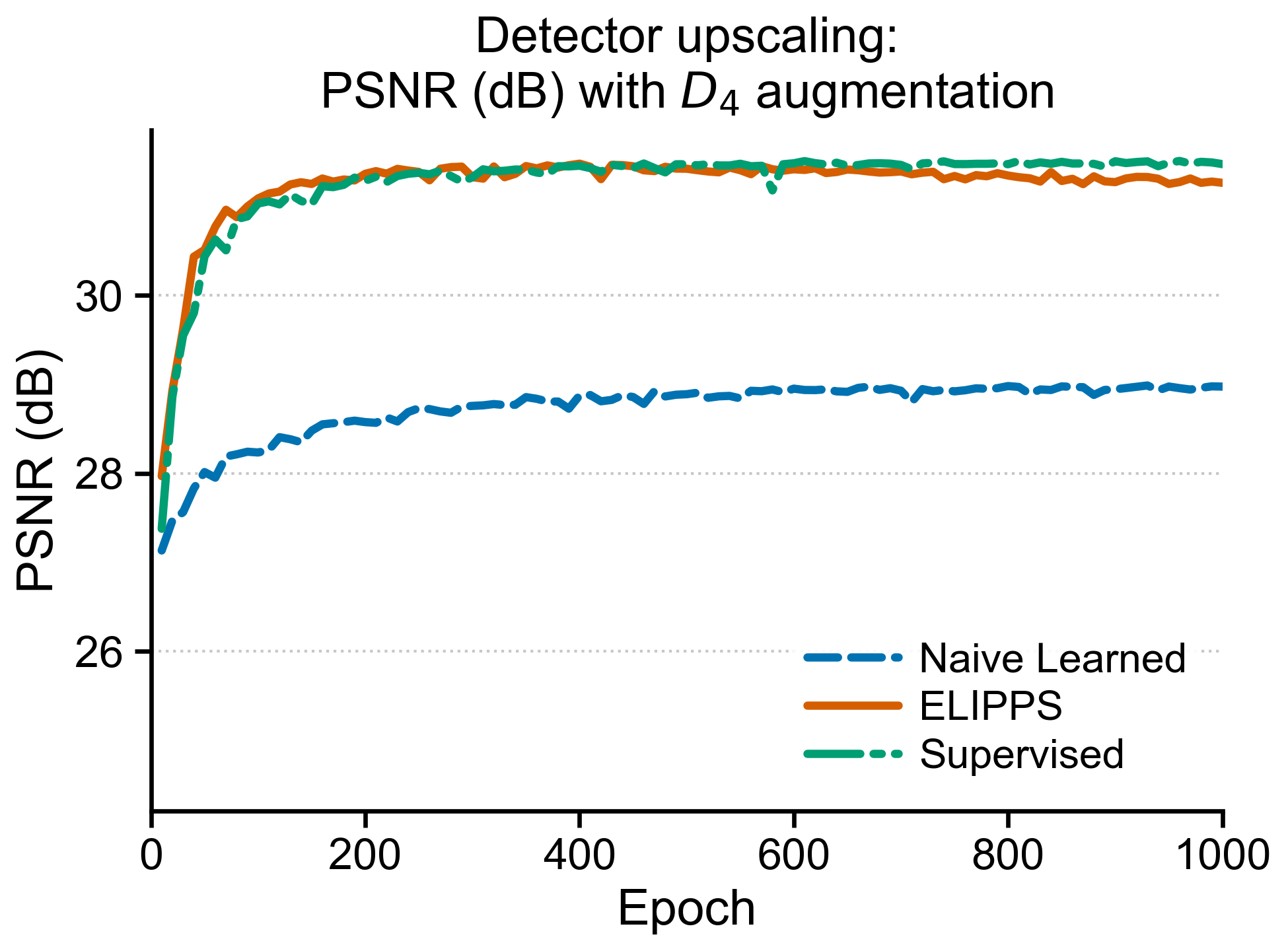}
        \caption{PSNR.}
        \label{fig:detector_psnr_d4}
    \end{subfigure}
    \hfill
    \begin{subfigure}[t]{0.48\linewidth}
        \centering
        \includegraphics[width=\linewidth]{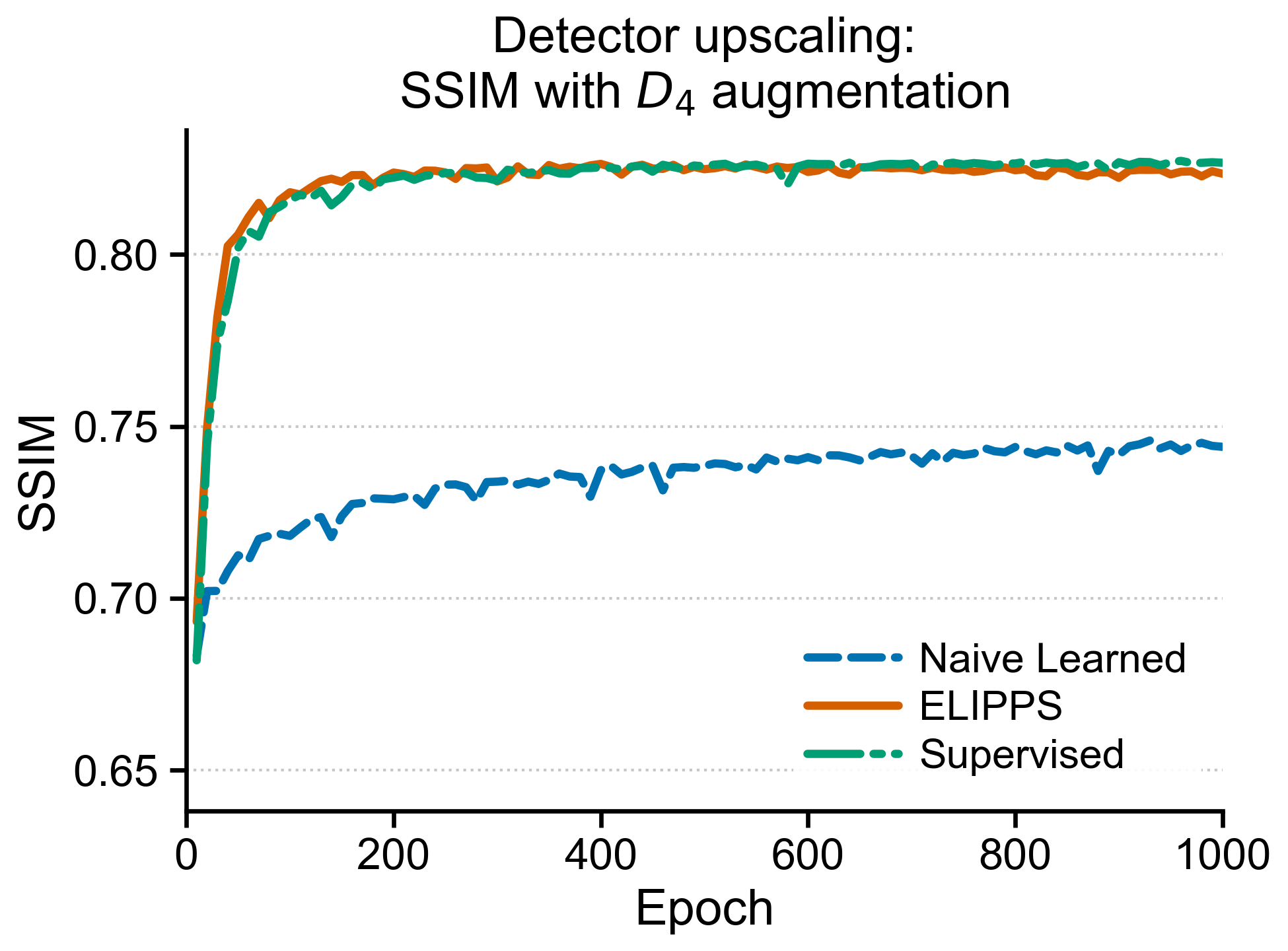}
        \caption{SSIM.}
        \label{fig:detector_ssim_d4}
    \end{subfigure}

    \caption{Test-set performance versus training epoch for detector
    upscaling with $D_4$ augmentation.}
    \label{fig:d4_detector_metrics}
\end{figure}

\begin{figure}[htb!]
    \centering

    \begin{subfigure}[t]{0.48\linewidth}
        \centering
        \includegraphics[width=\linewidth]{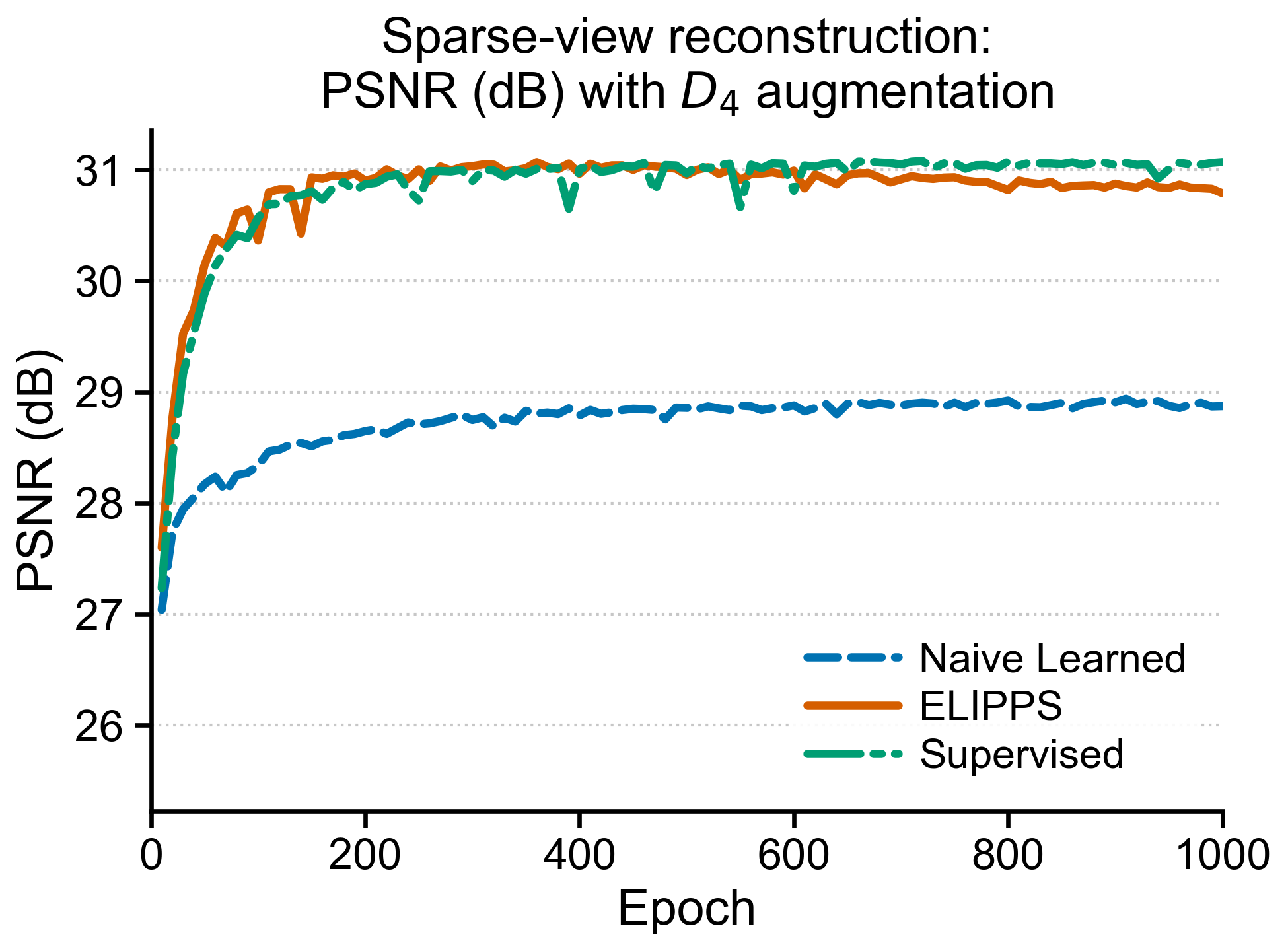}
        \caption{PSNR.}
        \label{fig:sparse_psnr_d4}
    \end{subfigure}
    \hfill
    \begin{subfigure}[t]{0.48\linewidth}
        \centering
        \includegraphics[width=\linewidth]{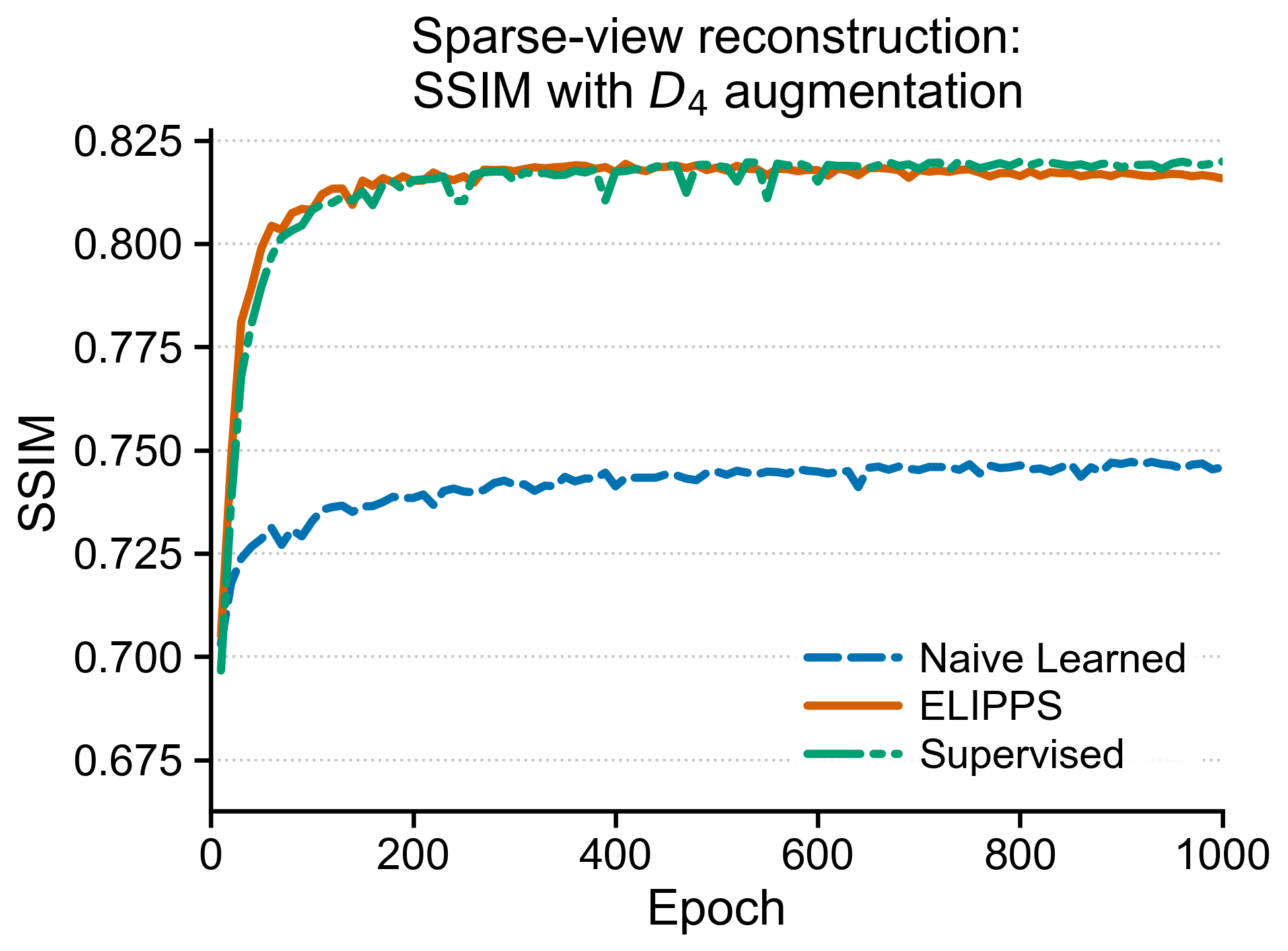}
        \caption{SSIM.}
        \label{fig:sparse_ssim_d4}
    \end{subfigure}

    \caption{Test-set performance versus training epoch for sparse-view
    reconstruction with $D_4$ augmentation.}
    \label{fig:d4_sparse_metrics}
\end{figure}

Figure~\ref{fig:combined_visual_results} shows qualitative comparisons: relative to FBP and naive refinement, ELIPPS produces fewer artifacts and sharper structures.

\begin{table}[htb!]
\centering
\caption{Quantitative comparison (PSNR in dB, SSIM) on the test set for
Detector Upscaling and Sparse-View Reconstruction. All learned methods are
reported at the final evaluation after 1000 training epochs. Entries are
mean $\pm$ sample standard deviation over the 35 test images.}
\label{tab:combined_results}
\begin{tabular}{lccccc}
\toprule
& \multicolumn{2}{c}{\textbf{Detector Upscaling}}
& \phantom{a}
& \multicolumn{2}{c}{\textbf{Sparse-View Recon.}} \\
\cmidrule(lr){2-3} \cmidrule(lr){5-6}
\textbf{Method}
& \textbf{PSNR $\uparrow$}
& \textbf{SSIM $\uparrow$}
&& \textbf{PSNR $\uparrow$}
& \textbf{SSIM $\uparrow$} \\
\midrule
FBP
& $24.34 \pm 2.09$ & $0.497 \pm 0.074$
&& $16.19 \pm 2.22$
& $0.286 \pm 0.030$ \\
Naive Learned
& $28.97 \pm 2.14$ & $0.744 \pm 0.084$
&& $28.87 \pm 2.18$ & $0.746 \pm 0.083$ \\
Naive Learned + $D_4$ TTA
& $29.77 \pm 2.14$ & $0.770 \pm 0.082$
&& $29.49 \pm 2.25$ & $0.765 \pm 0.085$ \\
ELIPPS
& $\underline{31.26 \pm 2.45}$ & $\underline{0.823 \pm 0.093}$
&& $\underline{30.79 \pm 2.55}$ & $\underline{0.816 \pm 0.095}$ \\
Supervised
& {\boldmath$31.47 \pm 2.50$} & {\boldmath$0.827 \pm 0.094$}
&& {\boldmath$31.07 \pm 2.60$} & {\boldmath$0.820 \pm 0.095$} \\
\bottomrule
\end{tabular}
\end{table}

\begin{figure}[htb!]
    \centering
    \includegraphics[width=\linewidth]{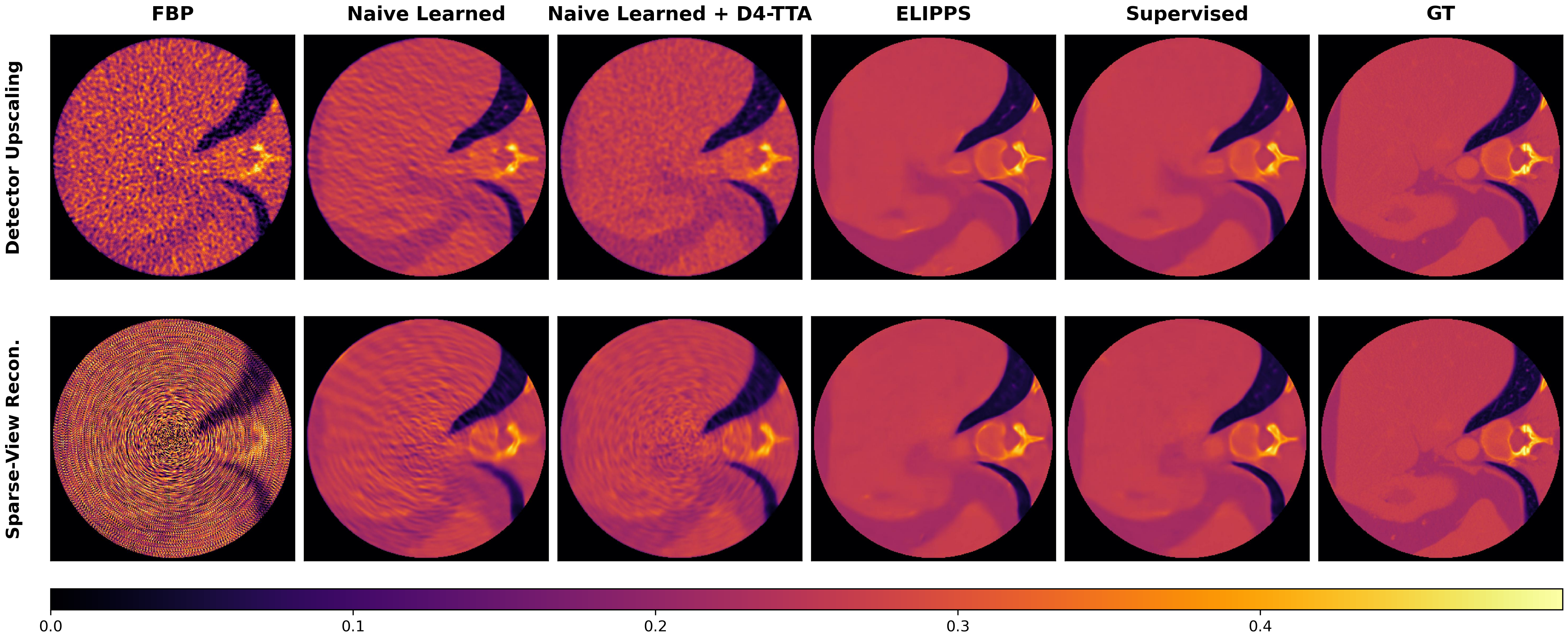}
    \caption{Visual comparison for Detector Upscaling (top) and Sparse-View Reconstruction (bottom).}
    \label{fig:combined_visual_results}
\end{figure}

\subsection{Ablation Study: Effect of $D_4$ Data Augmentation}
\label{sec:ablation}

The theory of Section~\ref{sec:selfsup} requires the $D_4$-invariance of the data distribution, which we enforced above by explicit $D_4$ augmentation. To assess the role of this augmentation, we repeat the training on the same $1\%$ LoDoPaB subset \emph{without} $D_4$ augmentation. In this setting, the invariance assumption \ref{cond:jointinv} is no longer enforced exactly, and the network must rely on the architectural equivariance and the physics-driven loss alone to generalize across unobserved regions. Training runs for $2000$ epochs and test evaluation is performed every $10$ epochs. The budget therefore differs from the $1000$ epochs used for the main results, so the two tables are not a like-for-like comparison; the ablation should be read as evidence that the architectural prior remains effective without symmetry-balanced data, not as a measurement of the size of the effect.

Performance remains stable (Table~\ref{tab:combined_results-NoAug}): relative to the $D_4$-augmented setting ELIPPS loses $0.52$~dB for detector upscaling and $0.81$~dB for sparse-view reconstruction, with gaps to the reference of $0.19$ and $0.44$~dB, while still outperforming the naive baseline by a wide margin. Architectural equivariance therefore acts as a regularizer even when the training data are not symmetry-balanced.

Figures~\ref{fig:nod4_detector_metrics} and \ref{fig:nod4_sparse_metrics} qualify these numbers. Without augmentation the ELIPPS curves peak at roughly $900$ to $1200$ epochs and then decline slowly, whereas the reference stays flat, so the fixed budget of $2000$ epochs lies past the maximum and part of the gap above follows from the budget rather than from the method. We report the final-epoch value nonetheless, since with no validation set selecting the peak would be selection on the test set. We do not interpret the decline: these are single runs on $358$ images, and ordinary overfitting is as plausible as any consequence of the violated \ref{cond:jointinv}.

\begin{figure}[htb!]
    \centering

    \begin{subfigure}[t]{0.48\linewidth}
        \centering
        \includegraphics[width=\linewidth]{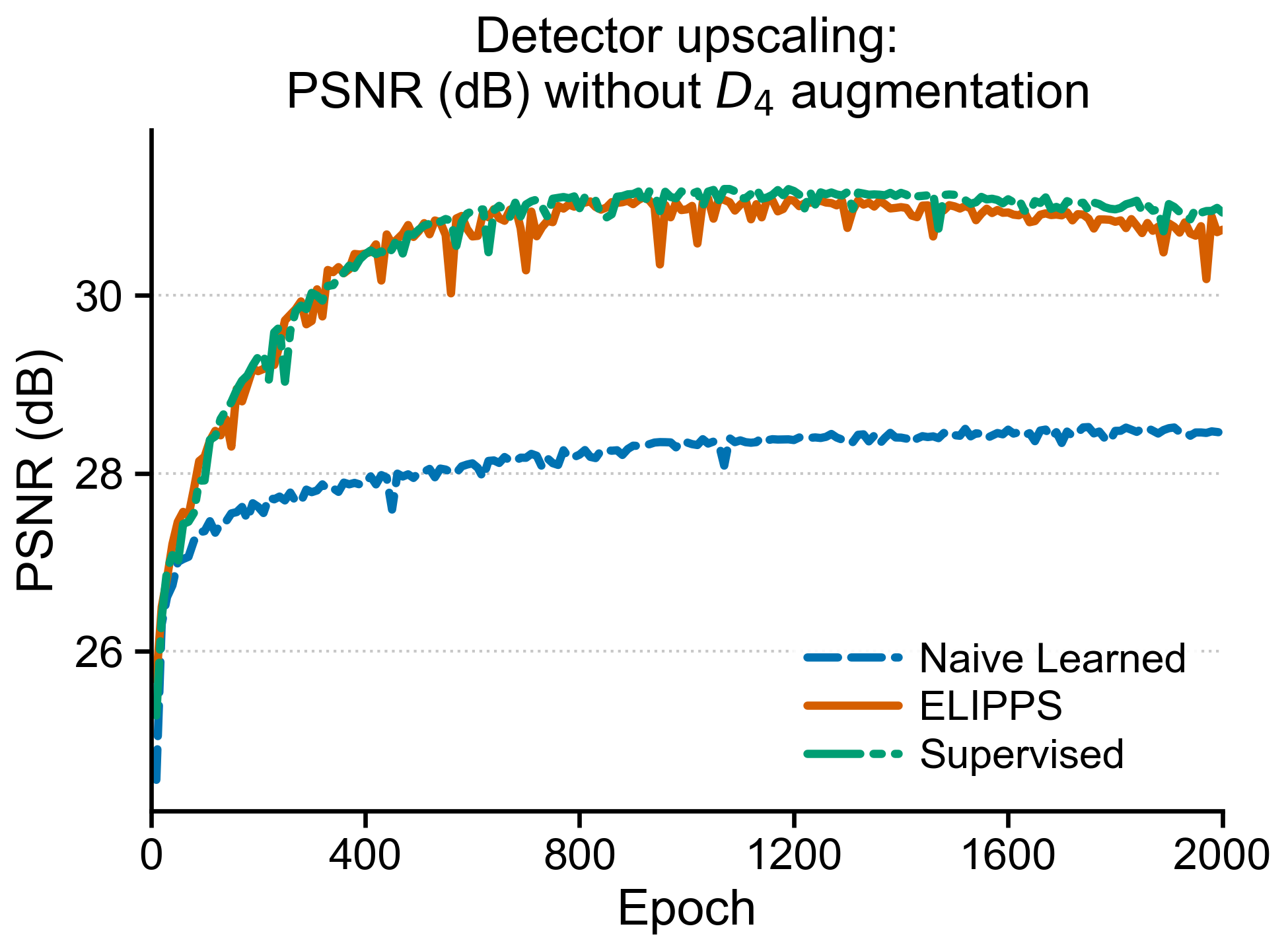}
        \caption{PSNR.}
        \label{fig:detector_psnr_nod4}
    \end{subfigure}
    \hfill
    \begin{subfigure}[t]{0.48\linewidth}
        \centering
        \includegraphics[width=\linewidth]{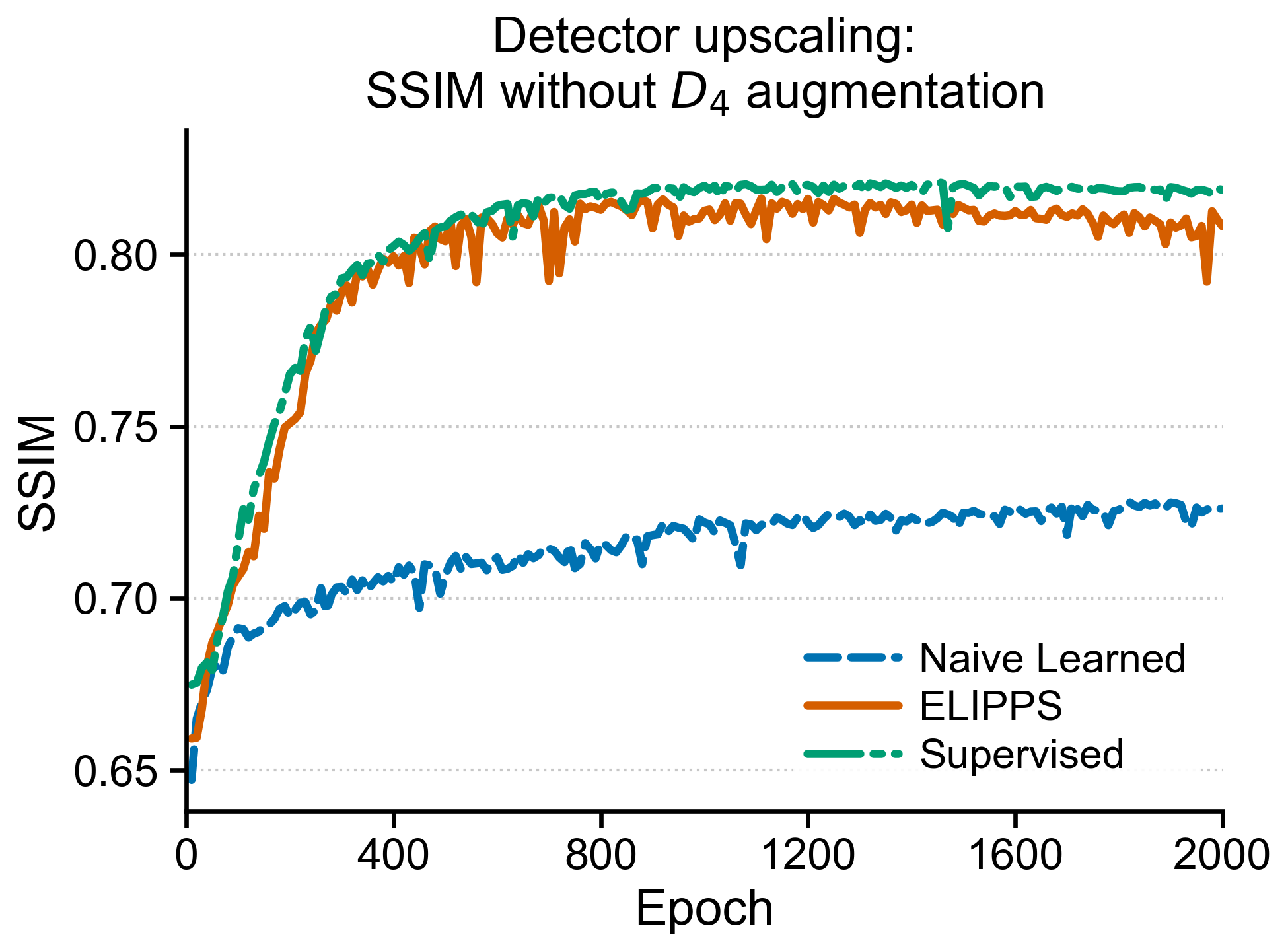}
        \caption{SSIM.}
        \label{fig:detector_ssim_nod4}
    \end{subfigure}

    \caption{Test-set performance versus training epoch for detector
    upscaling without $D_4$ augmentation.}
    \label{fig:nod4_detector_metrics}
\end{figure}

\begin{figure}[htb!]
    \centering

    \begin{subfigure}[t]{0.48\linewidth}
        \centering
        \includegraphics[width=\linewidth]{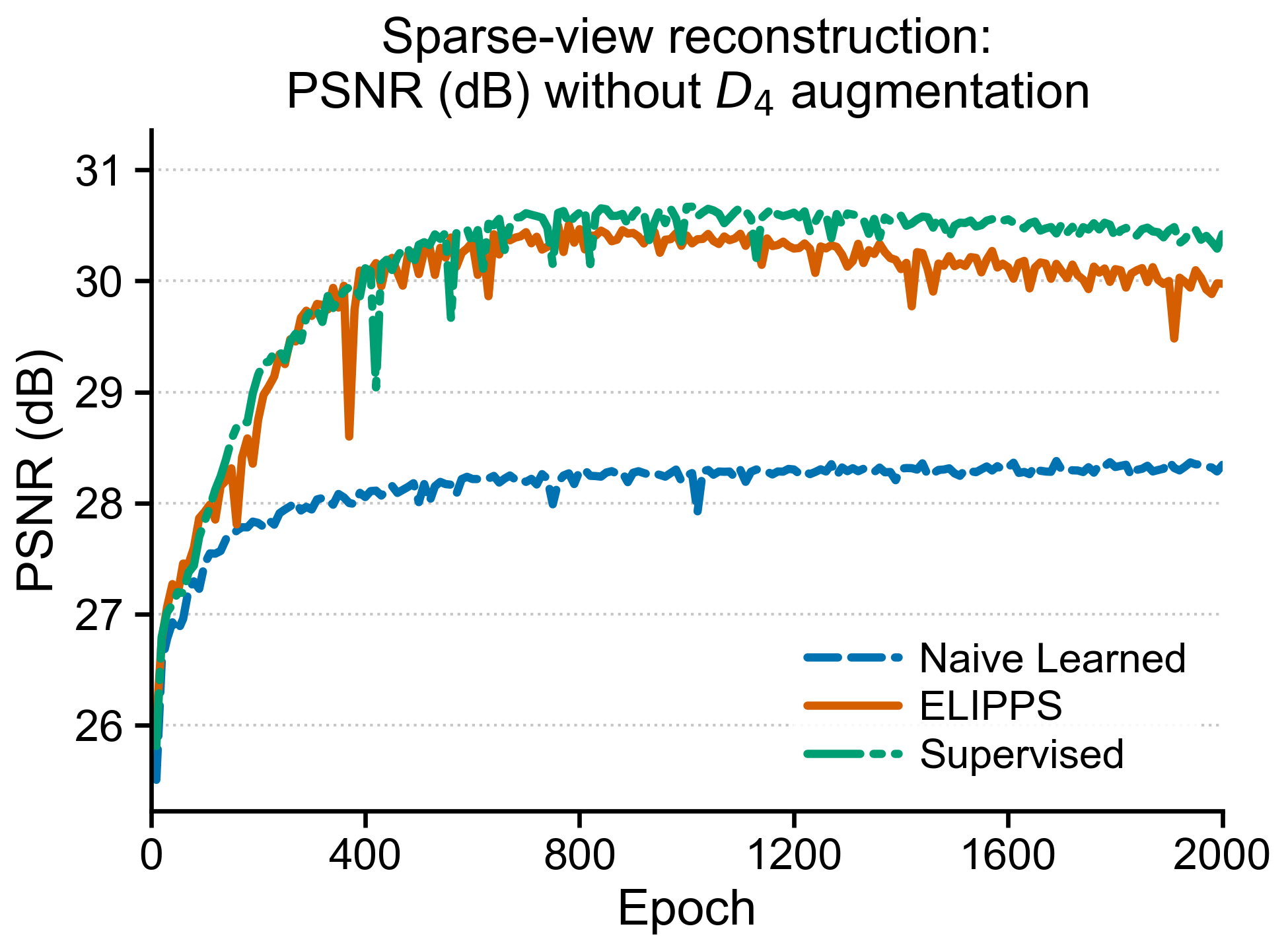}
        \caption{PSNR.}
        \label{fig:sparse_psnr_nod4}
    \end{subfigure}
    \hfill
    \begin{subfigure}[t]{0.48\linewidth}
        \centering
        \includegraphics[width=\linewidth]{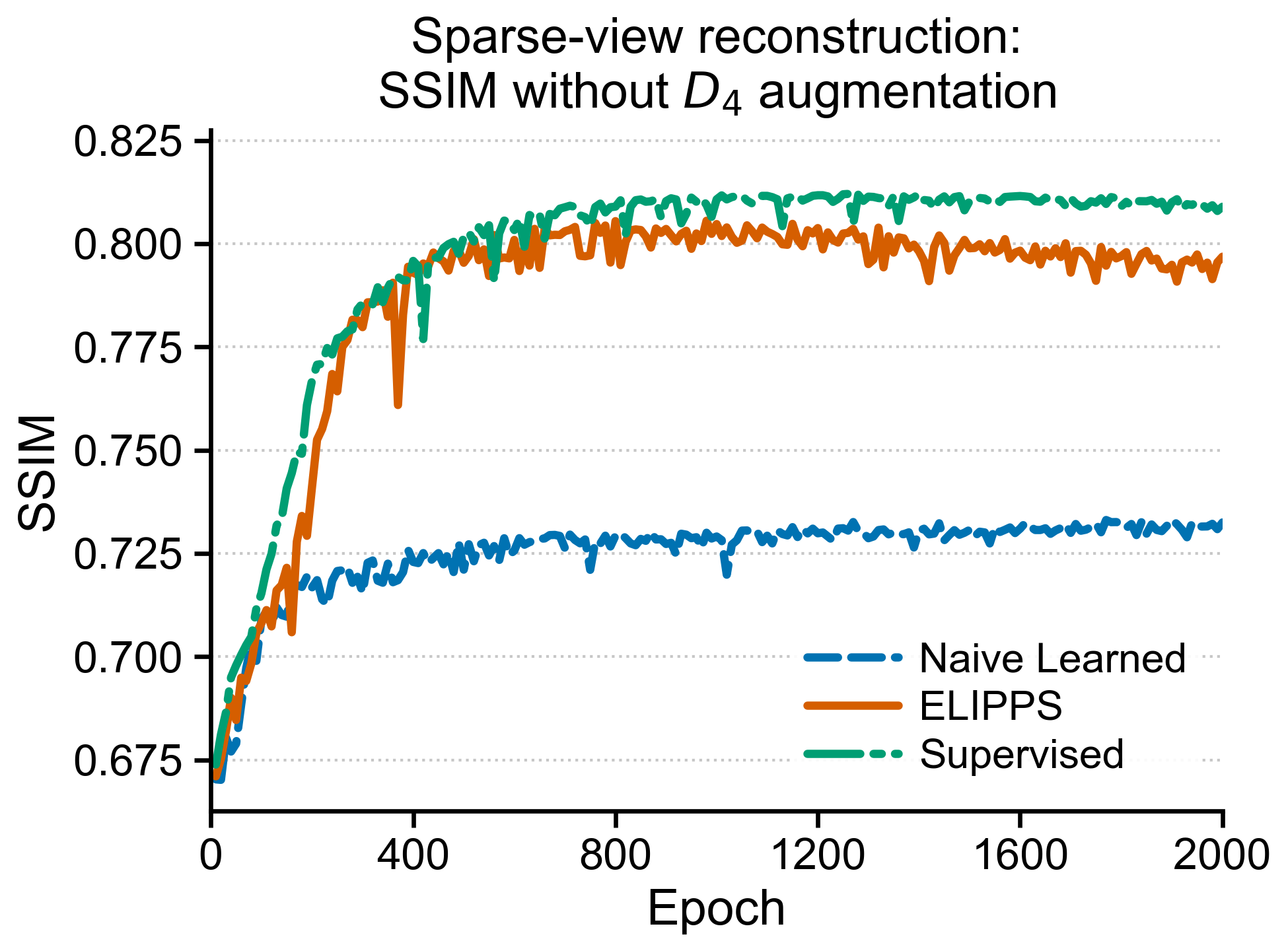}
        \caption{SSIM.}
        \label{fig:sparse_ssim_nod4}
    \end{subfigure}

    \caption{Test-set performance versus training epoch for sparse-view
    reconstruction without $D_4$ augmentation.}
    \label{fig:nod4_sparse_metrics}
\end{figure}

Figure~\ref{fig:combined_visual_results-NoAug} corroborates this: the equivariant model shows reduced aliasing and preserves structural detail.

\begin{table}[htb!]
\centering
\caption{Ablation without $D_4$ data augmentation (PSNR in dB, SSIM).
All learned methods are reported at the final evaluation after 2000 training
epochs. Entries are mean $\pm$ sample standard deviation over the 35 test
images.}
\label{tab:combined_results-NoAug}
\begin{tabular}{lccccc}
\toprule
& \multicolumn{2}{c}{\textbf{Detector Upscaling}}
& \phantom{a}
& \multicolumn{2}{c}{\textbf{Sparse-View Recon.}} \\
\cmidrule(lr){2-3} \cmidrule(lr){5-6}
\textbf{Method}
& \textbf{PSNR $\uparrow$}
& \textbf{SSIM $\uparrow$}
&& \textbf{PSNR $\uparrow$}
& \textbf{SSIM $\uparrow$} \\
\midrule
FBP
& $24.34 \pm 2.09$
& $0.497 \pm 0.074$
&& $16.19 \pm 2.22$
& $0.286 \pm 0.030$ \\
Naive Learned
& $28.46 \pm 2.11$
& $0.726 \pm 0.081$
&& $28.34 \pm 2.13$
& $0.733 \pm 0.081$ \\
Naive Learned + $D_4$ TTA
& $29.15 \pm 2.10$
& $0.751 \pm 0.080$
&& $28.86 \pm 2.21$
& $0.750 \pm 0.083$ \\
ELIPPS
& $\underline{30.74 \pm 2.36}$
& $\underline{0.808 \pm 0.090}$
&& $\underline{29.98 \pm 2.37}$
& $\underline{0.797 \pm 0.092}$ \\
Supervised
& {\boldmath$30.93 \pm 2.40$}
& {\boldmath$0.819 \pm 0.093$}
&& {\boldmath$30.42 \pm 2.46$}
& {\boldmath$0.809 \pm 0.094$} \\
\bottomrule
\end{tabular}
\end{table}

\begin{figure}[htb!]
    \centering
    \includegraphics[width=\linewidth]{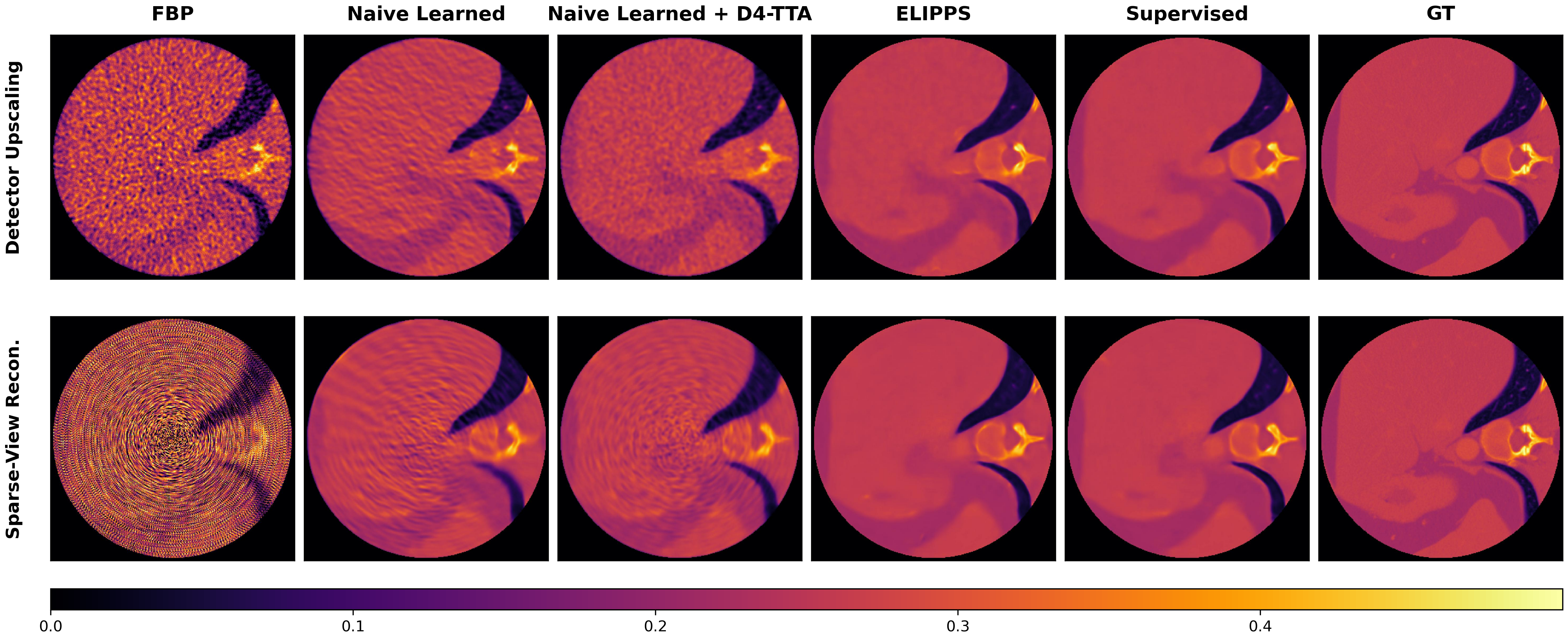}
    \caption{Visual comparison without D4 augmentation for Detector Upscaling (top) and Sparse-View Reconstruction (bottom).}
    \label{fig:combined_visual_results-NoAug}
\end{figure}

\subsection{Scope of the experimental study}
\label{sec:scope}

Section~\ref{sec:experiments} establishes a principle rather than a benchmark ranking: a reconstruction network can be trained from a fixed, incomplete supervision set alone and still reach the accuracy of a model trained with full clean measurements. The two instantiations were chosen so that this effect is visible rather than entangled with architecture search, and together with the reference points that bracket the setting they show that supervision on one eighth of the sinogram suffices.

We are not aware of another method that learns from partial self-supervision in the sense of
\eqref{eq:partial}, that is, from a supervision set that is fixed across acquisitions. Existing
approaches either require ground-truth images or fully sampled measurements, or they learn from
several distinct measurement operators \cite{tachella2022unsupervised}, or from the
undersampled measurement $y_I$ alone, without any supervision set; the latter, including equivariant
imaging and equivariant splitting, operate under the opposite requirement on the symmetry group
discussed in Remark~\ref{rem:opposite} and would, in the regime considered here, be deprived of the
mechanism they rely on.  A comparison with either family therefore answers a different question. The closest relative is \cite{Walder2025}, which the present framework generalizes from direct, noise-free observation to indirect, noisy and undersampled measurements.

Two aspects delimit what these numbers support. The study uses $358$ training and $35$ test images from a $1\%$ subset of LoDoPaB with one run per configuration, which fixes the resolution at which the comparison is meaningful; and the input and supervision sinograms come from two conditionally independent Poisson realizations, costing between $1.25$ and $2$ times the deployment scan depending on the task, so the protocol demonstrates the learning principle rather than a dose-neutral single-scan implementation, for which Poisson thinning is the natural variant.

Two properties of the construction are design principles rather than artefacts of this data set, and carry over to any modality whose forward operator intertwines with a finite symmetry group. The supervision set enters through its orbit, so what makes learning from one eighth of the measurements possible is the covering property of Proposition~\ref{prop:characterization}; $S$ is chosen by design and, by the same proposition, cannot be smaller, the eighth being attainable because the weighted mask makes the action effectively free. And since the training set is closed under the group action, Proposition~\ref{prop:orbit} applies sample by sample: the masked loss minimized here is, for every equivariant candidate, the full unmasked loss on the original images.

\section{Conclusion}

We introduced ELIPPS, a self-supervised training paradigm for undersampled inverse problems that
requires neither ground-truth images nor fully sampled measurements: training uses only
undersampled inputs together with incomplete measurements on a fixed supervision set. Two
structural properties render this seemingly ill-posed problem well-posed, invariance of the data
distribution under a family of transformations and intertwining of these transformations with the
forward operator. Over the induced equivariant class the partially self-supervised risk is
proportional to the full self-supervised risk, equivariance can be enforced architecturally at no
cost when the initial reconstruction is a least-squares right-inverse of the subsampled forward
operator, and the minimizer over the trained architecture class is the conditional expectation,
which in the image domain is the minimum mean squared error estimator up to $\ker(A)$. Over an
arbitrary equivariant class the two risks have the same minimizers, so misspecification causes an
approximation error but never an inconsistent training target; and uneven coverage leaves the conditional expectation identified exactly wherever the transformed supervision sets reach at all. For CT with the dihedral group $D_4$, supervision on one eighth of the sinogram
brings both detector upscaling and sparse-view reconstruction to the same order of accuracy as a
reference model trained on full clean measurements.

Several directions remain. The equivalence of the two risks requires only conditionally centered
supervision noise, which covers pre-log Poisson statistics exactly and the log-transformed count
model up to a quantifiable bias; the analysis of the factorized architecture, in contrast, uses an
independence structure that is exact only for i.i.d.\ Gaussian noise, and the FBP is not a
least-squares right-inverse. Quantitative estimates for approximate right-inverses and for
approximate independence are an open question. The conditional centering is also what forces the
two acquisitions, and blind-spot architectures achieve the same independence from a single
measurement, at the price of a network constraint that would have to be made compatible with the
equivariance imposed here. Further directions include nonlinear forward operators, learned initial
reconstructions, continuous symmetry groups, and other modalities such as magnetic resonance
imaging or photoacoustic tomography.

\section*{Acknowledgements}

\ifanonymous\else
The authors thank the developers of the LoDoPaB-CT benchmark for making the dataset publicly
available. This research received no specific grant from any funding agency in the public,
commercial or not-for-profit sectors.
\fi

The authors used Anthropic Claude (Opus 5, Fable 5 and Fable 5.1) and OpenAI ChatGPT (GPT-5.5 and
GPT-5.6) for language editing and reformulation, for suggestions on the organization and
presentation of the manuscript, for plausibility checks of the exposition and of individual proof
steps, for drafting parts of the manuscript, for assistance with the implementation of the code and
with the generation of the figures, and in a small number of places for a first formulation of a
statement and its proof. All such material was reworked and verified in full by the authors, as were
all bibliographic entries, which were checked against the primary sources. All definitions,
statements, proofs and numerical results are the responsibility of the authors, who have verified
each of them and take full responsibility for the content of this work.

\ifanonymous\else

\section*{Author contributions}

B.\ Walder contributed to the theoretical development and drafted parts of the manuscript.
M.\ Haltmeier supervised the project, guided the theoretical development, and wrote the manuscript
in its present form. L.\ Neumann contributed to the conception of the study, to the motivation of
the setting, and to the theoretical presentation. N.\ Gruber contributed to the initial development
of the method and to the presentation of the numerical section. G.\ Hwang implemented the software
and the numerical experiments, contributed to the theoretical development, and wrote the first draft
of the manuscript. All authors commented on the manuscript and approved the final version.
\fi

\section*{Data and code availability}

The LoDoPaB-CT dataset used in this study is publicly available \cite{leuschner2021lodopab}. The
code implementing the proposed method and reproducing all reported experiments will be made
publicly available in a repository upon publication.

\bibliographystyle{unsrt}
\bibliography{references}

\end{document}